\documentclass[11pt]{article}

\usepackage{amssymb,amsmath,amsfonts,amsthm,graphicx}

\usepackage{caption2}

\renewcommand{\subsection}{\subsubsection}

\newtheorem{theorem}{Theorem}[section]

\newtheorem{lemma}{Lemma}[section]
\newtheorem{proposition}{Proposition}[section]

\newtheorem{remark}{Remark}[section]

\def\diag{{\rm diag}\,}

\def\nt{|\hspace{-0.7pt}|\hspace{-0.7pt}|}

\def\div{{\rm div}\,}

\def\H{\mathcal H }

\def\j{\mathfrak J}

\newcommand{\vH}{\mathcal {H}} 
\newcommand{\pr}{\partial_}
\newcommand{\vfi}{\varphi}

\newcommand{\wh}{\widehat}   
\newcommand{\ds}{\displaystyle}
\newcommand{\ve}{\varepsilon}

\newcommand{\xip}{\xi_{\mathrm{p}}}
\newcommand{\xiv}{\xi_{\mathrm{v}}}
\newcommand{\dsq}[1]{\ds\sqrt{#1}}

\newcommand{\R}{{\mathbb R}}

\newcommand{\T}{{\mathbb T}}

\begin{document}

\title{\bf Influence of vacuum electric field on \\ the stability of a plasma-vacuum interface}

\author{{\bf Nikita Mandrik}\\
Novosibirsk State University, 2 Pirogova str., 630090, Novosibirsk, Russia\\
E-mail: manikitos@gmail.com
\and
{\bf Yuri Trakhinin}\\
Sobolev Institute of Mathematics, 4 Koptyug av., 630090, Novosibirsk, Russia\\
E-mail: trakhin@math.nsc.ru
}

\date{}

\maketitle

\begin{abstract}
We study the free boundary problem for the plasma-vacuum interface in ideal compressible magnetohydrodynamics.
Unlike the classical statement, when the vacuum magnetic field obeys the div-curl system of pre-Maxwell dynamics, we do not neglect the displacement current in the vacuum region and consider the Maxwell equations for electric and magnetic fields. We show that
a sufficiently large vacuum electric field can make the planar interface violently unstable.
We find and analyze a sufficient condition on the vacuum electric field that precludes violent instabilities. Under this condition satisfied at each point of the unperturbed nonplanar plasma-vacuum interface, we prove the well-posedness of the linearized problem in anisotropic weighted Sobolev spaces.

\end{abstract}

\section{Introduction}
\label{sec:1}

We consider the equations of ideal compressible magnetohydrodynamics (MHD):
\begin{subequations}
\begin{align}
& \partial_t\rho  +{\rm div}\, (\rho {v} )=0,\\
& \partial_t(\rho {v} ) +{\rm div}\,(\rho{v}\otimes{v} -{H}\otimes{H} ) +
{\nabla}q=0, \\
& \partial_t{H} -{\nabla}\times ({v} {\times}{H})=0,\label{1c}\\
& \partial_t\bigl( \rho e +{\textstyle \frac{1}{2}}|{H}|^2\bigr)+
{\rm div}\, \bigl((\rho e +p){v} +{H}{\times}({v}{\times}{H})\bigr)=0,
\end{align}
\label{1}
\end{subequations}
where $\rho$ denotes density, $v\in\mathbb{R}^3$ plasma velocity, $H \in\mathbb{R}^3$ magnetic field, $p=p(\rho,S )$ pressure, $q =p+\frac{1}{2}|{H} |^2$ total pressure, $S$ entropy, $e=E+\frac{1}{2}|{v}|^2$ total energy, and  $E=E(\rho,S )$ internal energy. With a state equation of gas, $\rho=\rho(p ,S)$, and the first principle of thermodynamics, \eqref{1} is a closed system for the unknown $ U =U (t, x )=(p, v,H, S)$.

System (\ref{1}) is supplemented by the divergence constraint
\begin{equation}
{\rm div}\, {H} =0
\label{2}
\end{equation}
on the initial data ${U} (0,{x} )={U}_0({x})$. As is known, taking into account \eqref{2}, we can easily symmetrize system \eqref{1} by rewriting it in the nonconservative form
\begin{equation}
\left\{
\begin{array}{l}
{\displaystyle\frac{\rho_{p}}{\rho}\,\frac{{\rm d} p}{{\rm d}t} +{\rm div}\,{v} =0,\qquad
\rho\, \frac{{\rm d}v}{{\rm d}t}-({H},\nabla ){H}+{\nabla}q  =0 ,}\\[9pt]
{\displaystyle\frac{{\rm d}{H}}{{\rm d}t} - ({H} ,\nabla ){v} +
{H}\,{\rm div}\,{v}=0},\qquad
{\displaystyle\frac{{\rm d} S}{{\rm d} t} =0},
\end{array}\right. \label{3}
\end{equation}
where ${\rm d} /{\rm d} t =\partial_t+({v} ,{\nabla} )$ and by $(\ ,\ )$ we denote the scalar product.
Equations (\ref{3}) form the symmetric system
\begin{equation}
\label{4}
A_0(U )\partial_tU+\sum_{j=1}^3A_j(U )\partial_jU=0
\end{equation}
which is hyperbolic if the matrix  $A_0= {\rm diag} \left(1/(\rho a^2),\rho ,\rho ,\rho , 1,1,1,1\right)$ is positive definite, i.e.,
\begin{equation}
\rho  >0,\quad \rho_p >0, \label{5}
\end{equation}
where $a=(\rho_p)^{-\frac{1}{2}}$ is the sound velocity and other symmetric matrices have the form
\[
A_1=\left( \begin{array}{cccccccc} \frac{v_1}{\rho a^2}&1&0&0&0&0&0&0\\[6pt]
1&\rho v_1&0&0&0&{H_2}&{H_3}&0\\
0&0&\rho v_1&0&0&-{H_1}&0&0\\ 0&0&0&\rho v_1&0&0&-{H_1}&0\\
0&0&0&0&{v_1}&0&0&0\\
0&{H_2}&-{H_1}&0&0&{v_1}&0&0\\
0&{H_3}&0&-{H_1}&0&0&{v_1}&0\\ 0&0&0&0&0&0&0&v_1\\
\end{array} \right) ,
\]
\[
A_2=\left( \begin{array}{cccccccc} \frac{v_2}{\rho a^2}&0&1&0&0&0&0&0\\[6pt]
0&\rho v_2&0&0&-{H_2}&0&0&0\\ 1&0&\rho v_2&0&{H_1}&0&{H_3}&0\\ 0&0&0&\rho
v_2&0&0&-{H_2}&0\\
0&-{H_2}&{H_1}&0&{v_2}&0&0&0\\
0&0&0&0&0&{v_2}&0&0\\
0&0&{H_3}&-{H_2}&0&0&{v_2}&0\\0&0&0&0&0&0&0&v_2
\end{array} \right) ,
\]
\[
A_3=\left( \begin{array}{cccccccc} \frac{v_3}{\rho a^2}&0&0&1&0&0&0&0\\[6pt]
 0&\rho v_3&0&0&-{H_3}&0&0&0\\ 0&0&\rho v_3&0&0&-{H_3}&0&0\\ 1&0&0&\rho
v_3&{H_1}&{H_2}&0&0\\
0&-{H_3}&0&{H_1}&{v_3}&0&0&0\\
0&0&-{H_3}&{H_2}&0&{v_3}&0&0\\
0&0&0&0&0&0&{v_3}&0\\ 0&0&0&0&0&0&0&v_3
\end{array} \right) .
\]

Plasma-vacuum interface problems for system \eqref{1} appear in the mathematical modeling of plasma confinement by magnetic fields. This subject is very popular since the 1950's, but most of theoretical studies were devoted to finding stability criteria of equilibrium states. The typical work in this direction is the famous paper of Bernstein et. al. \cite{BFKK} where the plasma-vacuum interface problem was considered in its classical statement modeling the plasma confined inside a perfectly conducting rigid wall and isolated from it by a vacuum region. In this statement (see also, e.g., \cite{Goed}) the plasma is described by the MHD equations \eqref{1} whereas in the vacuum region one considers the so-called {\it pre-Maxwell dynamics}
\begin{equation}
\nabla \times \mathcal{H} =0,\qquad {\rm div}\, \mathcal{H}=0\label{6}
\end{equation}
describing the vacuum magnetic field $\mathcal{H}\in\mathbb{R}^3$. That is, one neglects the displacement current $(1/c)\,\partial_tE$ not only
in nonrelativistic MHD but also in the Maxwell equations in vacuum, where  ${E}\in\mathbb{R}^3$ is the electric field and $c$ is the speed of the light. Then, from
\[
\nabla \times {E} =- \frac{1}{c}\,\partial_t\mathcal{H},\qquad {\rm div}\, E=0
\]
the vacuum electric field $E$ is a secondary variable that may be  computed from the magnetic field ${\mathcal H}$. Recall that the plasma electric field is a secondary variable as well because in ideal MHD
\begin{equation}
\label{El}
E^+=-\frac{1}{c}\,v\times H,
\end{equation}
where we use the notation $E^+$ to distinguish between the vacuum and plasma electric fields.\footnote{Below we will drop the subscript ``$-$'' for the vacuum electric field: $E:=E^-$.}

The classical statement \cite{BFKK,Goed} of the plasma-vacuum problem for systems \eqref{1} and \eqref{6} is closed by the boundary conditions
\begin{subequations}
\begin{align}
& \frac{{\rm d}F }{{\rm d} t}=0,\quad [q]=0,\quad  (H,N)=0 \label{7a}\\
& (\mathcal{H}, N)=0, \label{7b}
\end{align}
\label{7}
\end{subequations}
on the interface $\Gamma (t)=\{F(t,x)=0\}$ and the initial data
\begin{equation}
\label{8}
\begin{array}{ll}
{U} (0,{x})={U}_0({x}),\quad {x}\in \Omega^{+} (0),\qquad
F(0,{x})=F_0({x}),\quad {x}\in\Gamma(0) , \\
\mathcal{H}(0,x)=
\mathcal{H}_0(x),\quad {x}\in \Omega^{-}(0),
\end{array}
\end{equation}
for the plasma variable $U$, the vacuum magnetic field $\mathcal{H}$ and the function $F$, where $\Omega^+(t)$ and $\Omega ^-(t)$ are space-time domains occupied by the plasma and the vacuum respectively, $N=\nabla F$, and $[q]= q|_{\Gamma}-\frac{1}{2}|\mathcal{H}|^2_{|\Gamma}$ denotes the jump of the total pressure across the interface. The first condition in \eqref{7} means that the interface moves with the velocity of plasma particles at the boundary and since $F$ is an unknown, problem \eqref{1}, \eqref{6}--\eqref{8} is a free-boundary problem. Moreover, in the plasma confinement problem both the plasma and vacuum regions are bounded domains, and at the perfectly conducting rigid wall $\Sigma$ which is the exterior boundary of the vacuum region $\Omega^-(t)$ one states the standard boundary condition $(\mathcal{H}, n)=0$ (see \cite{YM}), where $n$ is a normal vector to $\Sigma$.

In astrophysics, the plasma-vacuum interface problem \eqref{1}, \eqref{6}--\eqref{8} can be used for modeling a star or the solar corona when magnetic fields are taken into account. In this case, the vacuum region surrounding a plasma body is usually assumed to be unbounded.

Until recently, there were no well-posedness results for full ({\it non-stationary}) plasma-vacuum models. A basic energy a priori estimate in Sobolev spaces for the linearization of the plasma-vacuum problem \eqref{1}, \eqref{6}--\eqref{8} was first derived in \cite{T10}, provided that the stability condition stating that the magnetic fields on either side of the interface are not collinear holds for a basic state (``unperturbed flow"). The existence of solutions to the linearized problem was then proved in \cite{ST1}. In \cite{T10,ST1}, as in \cite{Lind,T09.cpam}, it was assumed that the hyperbolicity conditions \eqref{5} are satisfied in $\Omega^+(t)$ up to the boundary $\Gamma(t)$, i.e., the density does not go to zero continuously, but has a jump (clearly, in the vacuum region $\Omega^-(t)$ the density is identically zero). It is noteworthy that this assumption is automatically satisfied for the uniform incompressible plasma, i.e., for the case when in problem \eqref{1}, \eqref{6}--\eqref{8} system \eqref{1} is replaced by the equations of ideal incompressible MHD with a uniform (constant) density. For this case the results analogous to those from \cite{T10,ST1} were recently obtained in \cite{MTT13}.

In \cite{T10,ST1}, for technical simplicity the moving interface  $\Gamma (t)$ was assumed to have the form of a graph $F=x_1-\varphi (t,x')$, $x'=(x_2,x_3)$, i.e., both the plasma and vacuum domains are unbounded. However, as was noted in the subsequent paper \cite{ST}, this assumption is not suitable in a pure form for the original nonlinear free boundary problem \eqref{1}, \eqref{6}--\eqref{8} because in that case the vacuum region $\Omega^-(t)=\{x_1<\varphi (t,x')\}$ is a simply connected domain.  Indeed, the elliptic problem \eqref{6}, \eqref{7b} has then only the trivial solution $\mathcal{H}=0$, and the whole problem is reduced to solving the MHD equations \eqref{1} with a vanishing total pressure $q$ on $\Gamma (t)$.

The technically difficult case of non simply connected vacuum regions was postponed in \cite{ST} to a future work. Instead of this, the plasma-vacuum system was assumed in \cite{ST} to be not isolated from the outside world due to a given surface current on the fixed boundary of the vacuum region that forces oscillations. Namely, in \cite{ST} the space domain $\Omega$ occupied by plasma and vacuum is given by
$
\Omega :=\{x\in\R^3 \; |  \, x_1\in(-1,1),\; x' \in \T^2\}  ,
$
where $\T^2$ denotes the $2$-torus, which can be thought of as the unit square with periodic boundary conditions, the interface $\Gamma$ is given by
$F=x_1-\varphi (t,x')=0$, and $\Omega^\pm(t)=\{x_1\gtrless \varphi(t,x')\}\cap\Omega$ are the plasma and vacuum domains respectively. On the fixed top and bottom boundaries $\Gamma_\pm := \{ (\pm 1,x') \, , \, x' \in \T^2 \}$ of the domain $\Omega$, one prescribed in \cite{ST} the boundary conditions
\begin{equation}
\label{9}
v_1=H_1=0 \quad {\rm on } \; [0,T] \times \Gamma_+ \,, \qquad
\nu\times\H=\j \quad {\rm on } \; [0,T] \times \Gamma_- \, ,
\end{equation}
where ${\nu}=(-1,0,0)$ is the outward normal vector at $\Gamma_-$ and  $\j$ represents a given surface current which forces oscillations onto the plasma-vacuum system. In laboratory plasmas this external excitation may be caused by a system of coils. This model can also be exploited for the analysis of waves in astrophysical plasmas, e.g., by mimicking the effects of excitation of MHD waves by an external plasma by means of a localized set of ``coils'', when the response of the internal plasma is the main issue (see a more complete discussion in \cite{Goed}).

Basing on the results of \cite{T10,ST1} for the linearized problem, under the above mentioned stability condition \cite{T10} satisfied at each point of the initial interface the existence and uniqueness of the solution to the nonlinear plasma-vacuum interface problem \eqref{1}, \eqref{6}--\eqref{9} in suitable anisotropic Sobolev spaces was recently proved in \cite{ST} by a suitable NashЦ-Moser-type iteration.

In relativistic settings, the displacement current $(1/c)\,\partial_tE$ cannot be neglected and we have the Maxwell equations
\begin{equation}
\frac{1}{c}\,\partial_t\mathcal{H}+\nabla\times E =0,\qquad \frac{1}{c}\,\partial_tE-\nabla\times \mathcal{H} =0\label{10}
\end{equation}
in the vacuum region whereas in the plasma region instead of system \eqref{1} one considers the equations of relativistic magnetohydrodynamics (RMHD). We do not include the equations
\begin{equation}
\div\mathcal{H}=0,\qquad \div E=0\label{11'}
\end{equation}
into the main system \eqref{10} because they are just divergence constraints on the initial data. The relativistic plasma-vacuum interface problem for the case of special relativity was first systematically studied in \cite{T12}. For technical simplicity the plasma and vacuum regions were assumed to be unbounded and given by $\Omega^\pm(t)=\{x_1\gtrless \varphi(t,x')\}$ respectively. The Maxwell equations \eqref{10} and the RMHD equations are supplemented by the interface conditions \eqref{7} and suitable boundary conditions for the vacuum electric field (see \cite{T12} and also below).  It should be noted that the relativistic version of the second boundary condition in \eqref{7a} has the form \cite{T12}
\begin{equation}
[q] =q|_{\Gamma}-\frac{1}{2}\left(|\mathcal{H}|^2-|E|^2\right)|_{\Gamma},
\label{11}
\end{equation}
with the relativistic total pressure $q=p +\frac{1}{2}\left(|H|^2-|E^+|^2\right)$, where the plasma electric field $E^+$ is given by \eqref{El}.

By considering particular cases for the unperturbed flow, it was shown in \cite{T12} that, unlike the non-relativistic case, even if the non-collinearity condition $(H\times\mathcal{H})|_{\Gamma}\neq 0$ from \cite{T10,ST1,ST} on the unperturbed magnetic fields holds a sufficiently large unperturbed vacuum electric field can make the relativistic planar interface violently unstable. The main result of \cite{T12} is finding a {\it sufficient stability condition} which gives a (basic) energy a priori estimate in the anisotropic weighted Sobolev space $H^1_*$ for the variable coefficients linearized problem for nonplanar plasma-vacuum interfaces (see \cite{Chen,YM,MST,Sec} and references therein as well as Section \ref{sec:3} for the definition of $H^m_*$). Namely, it was proved in \cite{T12} that under the non-collinearity condition the planar interface is stable if, roughly speaking, the unperturbed vacuum electric field is small enough. Moreover, if the sufficient stability condition holds at each point of the unperturbed nonplanar interface, then the linearized problem obeys the mentioned energy a priori estimate. The deduction of this a priori estimate is the first step towards the proof of a local-in-time existence and uniqueness theorem for the nonlinear problem.

However, due to enormous technical complication of the RMHD equations it is very difficult to analyze (even numerically) the parametric domain described by the sufficient stability condition found in \cite{T12} and how big is it in comparison with the whole stability domain, i.e., the whole domain of the well-posedness of the constant coefficients linearized problem for a planar interface.  The whole stability domain could be found by spectral analysis, but this seems technically impossible in practice (the linearized RMHD equations are rather complicated even for particular cases of the unperturbed flow).

At the same time, in the analysis in \cite{T12} relativistic effects play a rather passive role whereas the crucial influence on stability is exerted by vacuum electric field. Recall that in the classical statement of the non-relativistic plasma-vacuum interface problem \cite{BFKK,Goed,T10,ST1,ST} the influence of vacuum electric field is ignored because the vacuum electric field is a secondary variable defined through the vacuum magnetic field which should satisfy the div-curl system \eqref{6}. This seems reasonable at first sight (and, as we will see, this is indeed so if the vacuum electric field is small enough) because if we reduce the MHD system \eqref{3} and the Maxwell equations \eqref{10} to a dimensionless form by introducing the scaled values
\begin{equation}
\label{12}
\begin{array}{c}
\displaystyle
\tilde{x}=\frac{x}{\ell},\quad \tilde{t}=\frac{\bar{a}t}{\ell},\quad \tilde{v}=\frac{v}{\bar{a}},\quad \tilde{\rho}=\frac{\rho}{\bar{\rho}},\quad\tilde{p}=\frac{p}{\bar{\rho}\bar{a}^2},\quad
\widetilde{S}=\frac{S}{\bar{S}},\\[12pt]
\displaystyle
 \widetilde{H}=\frac{H}{\bar{a}\sqrt{\bar{\rho}}},\quad \widetilde{\mathcal{H}}=\frac{\mathcal{H}}{\bar{a}\sqrt{\bar{\rho}}},\quad
\widetilde{E}=\frac{E}{\bar{a}\sqrt{\bar{\rho}}},
\end{array}
\end{equation}
then after dropping tildes the MHD system in terms of the scaled values stays unchanged whereas the Maxwell equations take the form
\begin{subequations}
\label{13}
\begin{align}
& \varepsilon\partial_t\mathcal{H}+\nabla\times E =0, \label{13a}\\
& \varepsilon\partial_tE-\nabla\times \mathcal{H} =0 \label{13b}
\end{align}
\end{subequations}
(clearly, the divergence constraints \eqref{11'} stay unchanged), with
\[
\varepsilon =\frac{\bar{a}}{c}\,,
\]
where
$\ell$ is a characteristic length and $\bar{\rho}$, $\bar{a}$, $\bar{S}$ are constants associated with a uniform flow, namely, $\bar{a}$ is the sound speed, $\bar{\rho}$ is the density and $\bar{S}$ is the entropy for this flow. Since for non-relativistic speeds characteristic plasma velocities, in particular, the constant sound speed $\bar{a}$ are very small compared to the speed of light, the constant $\varepsilon$ is a very small but {\it fixed} parameter.

In this paper, unlike the classical statement in \cite{BFKK,Goed,T10,ST1,ST}, we do not set $\varepsilon =0$ in \eqref{13b} and consider the full Maxwell equations in vacuum. We show that in spite of the fact that $\varepsilon$ is a very small constant a large enough vacuum electric field crucially influence on the stability of a non-relativistic plasma-vacuum interface. Thus, in our new statement the non-relativistic plasma is still described by the MHD equations \eqref{3}\footnote{We assume that they are already written in terms of the scaled values \eqref{12} and tildes are dropped.} whereas the vacuum magnetic and electric fields obey the Maxwell equations \eqref{13}. As in \cite{T12}, for technical simplicity we consider the case of unbounded domains. Namely, we assume that the domains $\Omega^\pm(t)=\{x_1\gtrless \varphi(t,x')\}$ represent the plasma and vacuum regions respectively. On the interface $\Gamma (t)=\{F(t,x)=x_1-\varphi(t,x')=0\}$ we still have the boundary conditions \eqref{7}, where the jump $[q]$ is given by \eqref{11} and $q$ is the total pressure of non-relativistic plasma appearing in \eqref{1}, i.e., $q =p+\frac{1}{2}|{H} |^2$. That is, one has
\begin{subequations}
\label{16}
\begin{align}
& \partial_t \varphi= v_N, \label{16a}\\
& q= \textstyle{\frac{1}{2}}\left(|\mathcal{H}|^2-|E|^2\right)\qquad \mbox{on}\ \Gamma (t)\label{16b}
\end{align}
\end{subequations}
and
\begin{subequations}
\label{17}
\begin{align}
& H_N=0, \label{17a}\\
& \mathcal{H}_N=0\qquad \mbox{on}\ \Gamma (t),\label{17b}
\end{align}
\end{subequations}
where $v_N=(v,N)$, $H_N=(H,N)$, $\mathcal{H}_N =(\mathcal{H},N)$, and $N=\nabla F= (1,-\partial_2\varphi ,-\partial_3\varphi )$. As for current-vortex sheets \cite{T09}, conditions \eqref{17} are not real boundary conditions and should be considered as restrictions on the initial data.

The boundary conditions for the vacuum electric field are just jump conditions for equations \eqref{1c} (in a dimensionless form) and \eqref{13a}, i.e., for the conservation laws
\[
\partial_t(\varepsilon H^{\pm}) +\nabla \times E^{\pm} =0\quad \mbox{in}\ \Omega^{\pm}(t),
\]
with $H^+=H$, $H^-=\mathcal{H}$, $E^+=-\varepsilon( v\times H)$, cf. \eqref{El}, and $E^-=E$. These jump conditions have the known form \cite{BFKK}
\[
N\times [E]=\varepsilon\partial_t\varphi\,[H] \quad \mbox{on}\ \Gamma (t),
\]
with $[E]=E^+_{|\Gamma}-E|_{\Gamma}$, $[H]=H|_{\Gamma}-\mathcal{H}|_{\Gamma}$, and, taking into account \eqref{16a} and \eqref{17a}, we exclude from them the velocity and the plasma magnetic field:
\begin{equation}
\label{18}
N\times E =\varepsilon\partial_t\varphi\,\mathcal{H} \quad \mbox{on}\ \Gamma (t).
\end{equation}
The first  condition in \eqref{18} is nothing else than constraint \eqref{17b} and the rest two boundary conditions in \eqref{18} read
\begin{equation}
\label{19}
E_{\tau_2}=\varepsilon\mathcal{H}_3\partial_t\varphi,\quad
E_{\tau_3}=-\varepsilon\mathcal{H}_2\partial_t\varphi \qquad \mbox{on}\ \Gamma (t),
\end{equation}
with $E_{\tau_i}=E_1\partial_i\varphi+E_i$, $i=2,3$. In fact, if in \eqref{19} we formally set $\varepsilon =1$, then we get the corresponding boundary conditions for the relativistic case in \cite{T12}, where the speed of the light was taken to be equal to unity.

Summarizing the above, we obtain the free boundary value problem for system \eqref{4} in $\Omega^+(t)$ and system \eqref{13} in $\Omega^-(t)$ with the boundary conditions \eqref{16} and \eqref{19} on $\Gamma (t)$ and the initial  data
\begin{equation}
\begin{array}{c}
{U} (0,{x})={U}_0({x}),\quad {x}\in \Omega^{+} (0),\quad
{V} (0,{x})={V}_0({x}),\quad {x}\in \Omega^{-} (0),\\[3pt]
\varphi (0,{x}')=\varphi_0({x}'),\quad {x}\in\mathbb{R}^2,
\end{array}
\label{20}
\end{equation}
where $V=(\mathcal{H},E)$. Moreover, exactly as in \cite{T12}, we can prove that \eqref{2}, \eqref{11'} and \eqref{17} are restrictions on the initial data \eqref{20}, i.e., if they are satisfied at the first moment $t=0$, then they hold for all $t>0$. System \eqref{13} is always hyperbolic and, as in \cite{T12}, we assume that the hyperbolicity condition \eqref{5} is satisfied up to the boundary of the domain $\Omega^+(t)$.

In this paper, we study the linearized problem associated to \eqref{4}, \eqref{13}, \eqref{16}, \eqref{19}, \eqref{20}. We first obtain a non-relativistic counterpart of the sufficient stability condition from \cite{T12}. Then, our main goal is to analyze it for particular cases and compare with the spectral stability condition that was technically impossible in relativistic settings in \cite{T12}. The spectral stability condition is nothing else than the Kreiss-Lopatinski condition \cite{Kreiss,Maj84} for the constant coefficients linearized problem for a planar interface. Even in our non-relativistic settings, for technical reasons we are not able to find this condition for the general case of the unperturbed flow, but we can fortunately analyze it both analytically and numerically for some particular cases.

Following \cite{T12}, we can derive an energy a priori estimate in $H^1_*$ for the variable coefficients linearized problem for nonplanar plasma-vacuum interfaces, provided that the sufficient stability condition holds at each point of the unperturbed interface. Moreover, the existence of solutions of the linearized problem in $H^1_*$ was not proved in \cite{T12} and we fill this gap for our non-relativistic version of the linearized problem. But, it is worth noting that the same arguments towards the proof of existence are still applicable in relativistic settings in \cite{T12}.

The rest of the paper is organized as follows. In Section \ref{sec:2}, we reduce the free boundary problem \eqref{4}, \eqref{13}, \eqref{16}, \eqref{19}, \eqref{20} to an initial-boundary value problem in a fixed domain and discuss properties of the reduced problem. In Section \ref{sec:3}, we obtain the linearized problem and formulate main results for it (see Theorem \ref{t1} and Proposition \ref{ppcase}). In Section \ref{sec:4}, we find the mentioned sufficient stability condition for a planar interface by the energy method applied to the linearized problem in the case of constant coefficients. Moreover, in Section \ref{sec:4} we analyze this stability condition for the particular case from Proposition \ref{ppcase}. In Section \ref{sec:5}, by considering particular cases of the unperturbed constant solution, we prove that the planar interface can be violently unstable, and we study both analytically and numerically the spectral stability condition in the particular from Proposition \ref{ppcase}.  At last, Section \ref{sec:6} is devoted to the proof of the well-posedness of the linearized problem under the sufficient stability condition satisfied at each point of the unperturbed nonplanar plasma-vacuum interface.

\section{Reduced problem in a fixed domain}
\label{sec:2}

We straighten the interface $\Gamma$ by using the same change of independent variables as in \cite{T09,T09.cpam,T10}. That is, the unknowns $U$ and $V$ being smooth in $\Omega^{\pm}(t)$ are replaced by the vector-functions
\[
\widetilde{U}(t,x ):= {U}(t,\Phi^+ (t,x),x'),\quad
\widetilde{V}(t,x ):= V(t,\Phi^- (t,x),x'),
\]
which are smooth in the half-space $\mathbb{R}^3_+$, where
\[
\Phi^{\pm}(t,x ):= \pm x_1+\Psi^{\pm}(t,x ),\quad \Psi^{\pm}(t,x ):= \chi (\pm x_1)\varphi (t,x').
\]
and $\chi\in C^{\infty}_0(\mathbb{R})$ equals to 1 on $[-1,1]$, and $\|\chi'\|_{L_{\infty}(\mathbb{R})}<1/2$. Here, we use the cut-off function $\chi$ to avoid assumptions about compact support of the initial data in our (future) nonlinear existence theorem.\footnote{In \cite{T12}, for technical simplicity the cut-off function was not introduced, i.e., the simplest change of variables with $\chi\equiv 1$ was used.} The above change of variable is admissible if $\partial_1\Phi^{\pm}\neq 0$. The latter is guaranteed, namely, the inequalities $\partial_1\Phi^+> 0$ and $\partial_1\Phi^-< 0$ are fulfilled, if we consider solutions for which $\|\varphi\|_{L_{\infty}([0,T]\times\mathbb{R}^2)}\leq 1$. This holds if,
without loss of generality, we consider the initial data satisfying $\|\varphi_0\|_{L_{\infty}(\mathbb{R}^2)}\leq 1/2$, and the time $T$ in our existence theorem is sufficiently small.

Dropping for convenience tildes in $\widetilde{U}$ and $\widetilde{V}$, we reduce \eqref{4}, \eqref{13}, \eqref{16}, \eqref{19}, \eqref{20}  to the initial boundary value problem
\begin{equation}
\mathbb{P}(U,\Psi^+ )=0\quad\mbox{in}\ [0,T]\times \mathbb{R}^3_+,\label{24}
\end{equation}
\begin{equation}
\mathbb{V}(V,\Psi^- )=0\quad\mbox{in}\ [0,T]\times \mathbb{R}^3_+,\label{25}
\end{equation}
\begin{equation}
\mathbb{B}(U,V,\varphi )=0\quad\mbox{on}\ [0,T]\times\{x_1=0\}\times\mathbb{R}^{2},\label{26}
\end{equation}
\begin{equation}
(U,V)|_{t=0}=(U_0,V_0)\quad\mbox{in}\ \mathbb{R}^3_+,\qquad \varphi|_{t=0}=\varphi_0\quad \mbox{in}\ \mathbb{R}^{2},\label{27}
\end{equation}
where $\mathbb{P}(U,\Psi^+ )=L(U,\Psi^+ )U$, $\mathbb{V}(V,\Psi^- )= {M}(\Psi^- )V$,
\[
L(U, \Psi^+ )=A_0(U)\partial_t +\widetilde{A}_1(U,\Psi^+ )\partial_1+A_2(U )\partial_2+A_3(U )\partial_3,
\]
\[
\widetilde{A}_1(U,\Psi^+ )=\frac{1}{\partial_1\Phi^+}\left(
A_1(U )-A_0(U)\partial_t\Psi^+-A_2(U)\partial_2\Psi^+ -A_3(U)\partial_3\Psi^+\right),
\]
\[
{M}(\Psi^- )=\varepsilon I\,\partial_t +\widetilde{B}_1(\Psi^- )\partial_1+B_2\partial_2+B_3\partial_3,
\]
\[
\widetilde{B}_1(\Psi^- )= \frac{1}{\partial_1\Phi^-}\left(B_1-\varepsilon I\,\partial_t\Psi^- -B_2\partial_2\Psi^-
-B_3\partial_3\Psi^- \right),\qquad \partial_1\Phi^{\pm}=\pm 1 +\partial_1\Psi^{\pm},
\]
\[
B_j=\begin{pmatrix}
0_3 & b_j \\
b_j^T & 0_3
\end{pmatrix},\quad j=1,2,3,\qquad
\mathbb{B}(U,V,\varphi )=\left(
\begin{array}{c}
v_N -\partial_t\varphi  \\[3pt] q -\frac{1}{2}(|\mathcal{H}|^2-|E|^2) \\[3pt]
E_{\tau_2}-\varepsilon\mathcal{H}_3\partial_t\varphi  \\[3pt]
E_{\tau_3}+\varepsilon\mathcal{H}_2\partial_t\varphi
\end{array}
\right),
\]
\[
b_1=\left(\begin{array}{ccc}
 0 & 0 & 0 \\
 0 & 0 & 1 \\
 0 & -1 & 0
\end{array} \right),\quad
b_2=\left(\begin{array}{ccc}
 0 & 0 & 1 \\
 0 & 0 & 0 \\
-1 & 0 & 0
\end{array} \right),\quad
b_3=\left(\begin{array}{cccccc}
 0 & -1 & 0 \\
 1 & 0 & 0 \\
 0 & 0 & 0
\end{array} \right),
\]
\[
{v}_{N}={v}_1-{v}_2\partial_2{\Psi}^{+}-{v}_3\partial_3{\Psi}^{+},\quad {E}_{\tau_i}=E_1\partial_i\Psi^{-}+E_i,\quad i=2,3.
\]

Regarding constraints \eqref{2}, \eqref{11'} and \eqref{17}, following \cite{T09} and \cite{T12}, we can proof the following propositions.

\begin{proposition}
Let the initial data \eqref{27} satisfy
\begin{equation}
\div  h=0
\label{25'}	
\end{equation}
and
\begin{equation}
H_{N}|_{x_1=0}=0,
\label{26'}
\end{equation}
where
\[
h=(H_N,H_2\partial_1\Phi^+,H_3\partial_1\Phi^+),\quad
H_N=H_1-H_2\partial_2\Psi^+-H_3\partial_3\Psi^+.
\]
If problem \eqref{24}--\eqref{27} has a solution $(U,V,\varphi )$, then this solution satisfies \eqref{25'} and \eqref{26'} for all $t\in [0,T]$.
\label{p1}
\end{proposition}

\begin{proposition}
Let the initial data \eqref{27} satisfy
\begin{equation}
\mathcal{H}_N|_{x_1=0}=0\label{15.2'}
\end{equation}
and
\begin{equation}
{\rm div} \,\mathfrak{h}=0,\quad {\rm div}\, \mathfrak{e} =0,\label{28}
\end{equation}
where
\[
\mathcal{H}_{N}=\mathcal{H}_1-\mathcal{H}_2\partial_2\Psi^{-}-\mathcal{H}_3\partial_3\Psi^{-},\quad \mathfrak{h}=(\mathcal{H}_{N},\mathcal{H}_2\partial_1\Phi^{-},\mathcal{H}_3\partial_1\Phi^{-}),
\]
\[
\mathfrak{e}=({E}_{N},{E}_2\partial_1\Phi^{-},{E}_3\partial_1\Phi^{-}),\quad {E}_{N}={E}_1-{E}_2\partial_2\Psi^{-}-{E}_3\partial_3\Psi^{-}.
\]
If problem \eqref{24}--\eqref{27} has a solution $(U,V,\varphi )$ with the property
\begin{equation}
\partial_t\varphi \leq 0,\label{expan1}
\end{equation}
then this solution satisfies \eqref{15.2'} and \eqref{28} for all $t\in [0,T]$.

If problem \eqref{24}--\eqref{27} with the two additional boundary conditions
\begin{equation}
{\rm div}\, \mathfrak{h}|_{x_1=0} =0\qquad\mbox{and}\qquad
{\rm div}\, \mathfrak{e}|_{x_1=0} =0,
\label{add.bound}
\end{equation}
has a solution $(U,V,\varphi )$ with the property
\begin{equation}
\partial_t\varphi > 0,\label{shrink}
\end{equation}
then this solution again satisfies \eqref{15.2'} and \eqref{28} for all $t\in [0,T]$.
\label{p2}
\end{proposition}

The proof of Proposition \ref{p1} is absolutely the same as the corresponding one in \cite{T09} for current-vortex sheets. To prove Proposition \ref{p2} we should just literally repeat arguments from \cite{T12} with technical modifications connected with the introduction of the cut-off function $\chi$ in our change of variables. As in \cite{T12}, without loss of generality, we will below consider only case \eqref{expan1} when the plasma expands into the vacuum.\footnote{As in \cite{T12}, in this paper we postpone to a future work the consideration of the general mixed case when $\partial_t\varphi$ is indefinite in sign. For this difficult case the Maxwell system \eqref{25} is of {\it variable multiplicity}, see Remark \ref{r1}.} Moreover, we naturally assume that the ``stable'' counterpart
\begin{equation}
\partial_t\varphi < 0\label{expan}
\end{equation}
of condition \eqref{expan1} holds. This means that in our future local-in-time existence theorem we will consider initial data satisfying condition \eqref{expan} and its fulfilment for solutions will be guaranteed by a small enough time of their existence.

Under assumption \eqref{expan} the boundary $x_1=0$ is noncharacteristic for system \eqref{25}. Indeed, all the eigenvalues of the boundary matrix
\begin{equation}
\mathfrak{B} (\varphi ):=\widetilde{B}_1({\Psi}^-)|_{x_1=0}=\left(
\begin{array}{cccccc}
\varepsilon\partial_t\varphi & 0 & 0& 0 & -\partial_3{\varphi} & \partial_2{\varphi} \\
0 & \varepsilon\partial_t\varphi & 0& \partial_3{\varphi} & 0 & 1 \\
0 & 0 & \varepsilon\partial_t\varphi& -\partial_2{\varphi} & -1 & 0 \\
0 & \partial_3{\varphi} & -\partial_2{\varphi}& \varepsilon\partial_t\varphi & 0 & 0 \\
-\partial_3{\varphi} & 0 & -1& 0 & \varepsilon\partial_t\varphi & 0 \\
\partial_2{\varphi} & 1 & 0& 0 & 0 & \varepsilon\partial_t\varphi
\end{array}
\right)
\label{bm}
\end{equation}
for system \eqref{25} are non-zero:
\[
\lambda_{1,2}= \varepsilon\partial_t\varphi +
\sqrt{1 +(\partial_2\varphi)^2+(\partial_3\varphi)^2},
\]
\[
\lambda_{3,4}= \varepsilon\partial_t\varphi-
\sqrt{1 +(\partial_2\varphi)^2+(\partial_3\varphi)^2},\qquad
\lambda_{5,6} =\varepsilon\partial_t\varphi .
\]
For case \eqref{expan}, the eigenvalues $\lambda_{1,2} >0$ and $\lambda_{k} <0$, $k=3,4,5,6$. As in \cite{ST1,ST,T10}, the boundary matrix $\mathfrak{A} (U_{|x_1=0},\varphi ):=\widetilde{A}_1(U,\Psi^+ )|_{x_1=0}$ for the MHD system \eqref{24} has one positive and one negative eigenvalue and the others are zero (see also next section). That is, the boundary $x_1=0$ is characteristic for system \eqref{24}, and the whole system \eqref{24}, \eqref{25} for $U$ and $V$ has three incoming characteristics. This means that the number of boundary conditions in \eqref{26} is correct because one of them is needed for determining the function $\varphi (t,x')$.

\begin{remark}
{\rm
For case \eqref{shrink}, when we have shrinkage of the plasma region, the number of incoming characteristics is five (for this case the eigenvalues $\lambda_{5,6}$ above are positive). It means that the correct number of boundary conditions is six and problem \eqref{24}--\eqref{27} is formally underdetermined because it is missing two boundary conditions. However, we supplement \eqref{26} with the additional boundary conditions \eqref{add.bound} which enable one to prove Proposition \ref{p2} for case \eqref{shrink}. For the opposite case \eqref{expan} these additional boundary conditions are unnecessary. This makes rather difficult the analysis of the general case when for some parts of the interface the plasma expands into the vacuum and for other parts the vacuum expands into the plasma. On the other hand, locally in space either condition \eqref{expan1} or \eqref{shrink} is satisfied on the interface and this makes reasonable the consideration of one of them.
}
\label{r1}
\end{remark}

\section{Linearized problem and main results}
\label{sec:3}

\subsection{Basic state}

Let
\begin{equation}
(\widehat{U}(t,x ),\widehat{V}(t,x ),\hat{\varphi}(t,{x}'))
\label{37}
\end{equation}
be a given sufficiently smooth vector-function with $\widehat{U}=(\hat{p},\hat{u},\widehat{H},\widehat{S})$, $\widehat{V}=(\widehat{\mathcal{H}},\widehat{E})$, and
\begin{equation}
\|(\widehat{U},\widehat{V})\|_{W^2_{\infty}(\Omega_T)}+
\|\partial_1(\widehat{U},\widehat{V})\|_{W^2_{\infty}(\Omega_T)}+
\|\hat{\varphi}\|_{W^3_{\infty}(\partial\Omega_T)} \leq K,
\label{38}
\end{equation}
where $K>0$ is a constant,
\[
\Omega_T:= (-\infty, T]\times\mathbb{R}^3_+,\quad \partial\Omega_T:=(-\infty ,T]\times\{x_1=0\}\times\mathbb{R}^{2},
\]
and below all the ``hat'' values will be related to the basic state \eqref{37}.

We assume that the basic state \eqref{37} satisfies the hyperbolicity condition \eqref{5} in $\overline{\Omega_T}$,
\begin{equation}
\rho (\hat{p},\widehat{S}) >0,\quad \rho_p(\hat{p},\widehat{S}) >0 ,
\label{39}
\end{equation}
the boundary conditions in \eqref{26} except the second one on $\partial\Omega_T$,
\begin{equation}
\hat{v}_N|_{x_1=0} =\varkappa ,\quad \widehat{E}_{\tau_2}|_{x_1=0}=
\varepsilon\varkappa\widehat{\mathcal{H}}_3|_{x_1=0} , \quad \widehat{E}_{\tau_3}|_{x_1=0}=-\varepsilon
\varkappa\widehat{\mathcal{H}}_2|_{x_1=0},
\label{40}
\end{equation}
and the interior equations for $\widehat{H}$ and $\widehat{\mathcal{H}}$ in ${\Omega_T}$ contained in \eqref{24} and \eqref{25}:
\begin{equation}
\partial_t\widehat{H}+\frac{1}{\partial_1\widehat{\Phi}^+}\left\{ (\hat{w} ,\nabla )
\widehat{H} - (\hat{h} ,\nabla ) \hat{v} + \widehat{H}{\rm div}\,\hat{u}\right\} =0,
\label{41}
\end{equation}
\begin{equation}
\varepsilon\partial_t\widehat{\mathcal{H}}-\varepsilon\frac{\partial_t\widehat{\Psi}^- }{\partial_1\widehat{\Phi}^-}\partial_1\widehat{\mathcal{H}}
+\frac{1}{\partial_1\widehat{\Phi}^-}\left(b_1- b_2\partial_2\widehat{\Psi}^--b_3\partial_3\widehat{\Psi}^- \right) \widehat{E}=0,
\label{42}
\end{equation}
where $\varkappa =\partial_t\hat{\varphi},\quad \hat{w}=\hat{u}- (\partial_t\widehat{\Psi}^+,0,0)$, $\hat{u}=(\hat{v}_N,\partial_1\widehat{\Phi}^+\hat{v}_2,\partial_1\widehat{\Phi}^+\hat{v}_3)$.
At last, we assume that the basic state satisfies condition \eqref{expan}. More precisely, let
\begin{equation}
\varkappa \leq -\varepsilon_1 <0
\label{expan.lin}
\end{equation}
on $\partial\Omega_T$ with a fixed constant $\varepsilon_1$.

Moreover, without loss of generality we assume that $\|\hat{\varphi}\|_{L_{\infty}(\partial\Omega_T)}<1$. This implies
\[
\partial_1\widehat{\Phi}^+\geq 1/2,\quad \partial_1\widehat{\Phi}^-\leq -  1/2.
\]
Note that (\ref{38}) yields
\[
\|\widehat{W}\|_{W^2_{\infty}(\Omega_T)} \leq C(K),
\]
where $\widehat{W}:=(\widehat{U},\partial_1\widehat{U},\widehat{\mathcal{H}},\nabla_{t,x}\widehat{\Psi}^+,
\nabla_{t,x}\widehat{\Psi}^-)$, $\nabla_{t,x}=(\partial_t, \nabla )$, and $C=C(K)>0$ is a constant depending on $K$.

Assumptions \eqref{41} and \eqref{42} on the basic state are necessary to deduce for the linearized problem equations associated to constraints \eqref{25'}--\eqref{28} (see \cite{T09,T12}). For the same reason we have to assume that the basic state satisfies these constraints:
\begin{equation}
\widehat{H}_N|_{x_1=0}=0,\quad \widehat{\mathcal{H}}_N|_{x_1=0}=0,\label{43}
\end{equation}
\begin{equation}
\div \hat{\sf h}=0,\quad {\rm div} \,\hat{\mathfrak{h}}=0,\quad {\rm div}\,\hat{\mathfrak{e}}=0.\label{44}
\end{equation}
However, together with \eqref{41} and \eqref{42}  it is enough to assume that \eqref{43} and \eqref{44} hold only for $t=0$. Then, it follows from the proofs of Propositions \ref{p1} and \ref{p2} (see \cite{T09,T12}) that conditions \eqref{43} are true on $\partial\Omega_T$ and Eqs. \eqref{44} hold in ${\Omega_T}$. Thus, without loss of generality for the basic state we also require the fulfillment of \eqref{43} on $\partial\Omega_T$ and \eqref{44} in $\Omega_T$.

At last, we note that the first conditions in \eqref{40} and \eqref{43} endow the boundary matrix
$\diag (\mathfrak{A} (U_{|x_1=0},\varphi ),\mathfrak{B}(\varphi))$
for the linearized problem with the same structure as for the boundary matrix for the nonlinear problem \eqref{24}--\eqref{27}. Concerning the second and third conditions in \eqref{40}, we  need them for deducing a linearized version of \eqref{15.2'} for solutions of the linearized problem.

\subsection{Linearized problem}

The linearized equations for \eqref{24}--\eqref{26}  read:
\[
\mathbb{P}'(\widehat{U},\widehat{\Psi}^{+})(\delta U,\delta\Psi^{+}):=
\frac{\rm d}{{\rm d}\epsilon}\mathbb{P}(U_{\epsilon},\Psi_{\epsilon}^{+})|_{\epsilon =0}=f_{\rm I}
\quad \mbox{in}\ \Omega_T,
\]
\[
\mathbb{V}'(\widehat{V},\widehat{\Psi}^{-})(\delta V,\delta\Psi^{-}):=
\frac{\rm d}{{\rm d}\epsilon}\mathbb{V}(V_{\epsilon},\Psi_{\epsilon}^{-})|_{\epsilon =0}=f_{\rm II}
\quad \mbox{in}\ \Omega_T,
\]
\[
\mathbb{B}'(\widehat{U},\widehat{V},\hat{\varphi})(\delta U,\delta V,\delta \varphi ):=
\frac{\rm d}{{\rm d}\epsilon}\mathbb{B}(U_{\epsilon},V_{\epsilon},\varphi_{\epsilon})|_{\epsilon =0}={g}
\quad \mbox{on}\ \partial\Omega_T,
\]
where $U_{\epsilon}=\widehat{U}+ \epsilon\,\delta U$, $V_{\epsilon}=
\widehat{V}+\epsilon\,\delta V$,
$\varphi_{\epsilon}=\hat{\varphi}+ \epsilon\,\delta \varphi$, and
\[
\Psi_{\epsilon}^{\pm}(t,x ):=\chi (\pm x_1)\varphi_{\epsilon}(t,x'),\quad
\Phi_{\epsilon}^{\pm}(t,x ):=\pm x_1+\Psi_{\epsilon}^{\pm}(t,x ),
\]
\[
\delta\Psi^{\pm}(t,x ):=\chi (\pm x_1)\delta \varphi (t,x ).
\]
Here we introduce the source terms $f=(f_{\rm I}, f_{\rm II})=(f_1,\ldots ,f_{14})$ and $g=(g_1,\ldots ,g_4)$ to make the interior equations and the boundary conditions inhomogeneous.

We easily compute the exact form of the linearized equations (below we drop $\delta$):
\[
\mathbb{P}'(\widehat{U},\widehat{\Psi}^{+})(U,{\Psi}^{+} )
=
L(\widehat{U},\widehat{\Psi}^{+})U +\mathcal{C}(\widehat{U},\widehat{\Psi}^{+})
U -   \bigl\{L(\widehat{U},\widehat{\Psi}^{+})\Psi^{+}\bigr\}\frac{\partial_1\widehat{U}}{\partial_1\widehat{\Phi}^+}=f_{\rm I},
\]
\[
\mathbb{V}'(\widehat{V},\widehat{\Psi}^{-})(V,{\Psi}^{-} )
=
M(\widehat{\Psi}^{-})V  +   \bigl\{M(\widehat{\Psi}^{-})\Psi^{-}\bigr\}\frac{\partial_1\widehat{V}}{\partial_1\widehat{\Phi}^-}=f_{\rm II},
\]
\[
\mathbb{B}'(\widehat{U},\widehat{V},\hat{\varphi})(U,V,\varphi )=
{\left(
\begin{array}{c}
v_{N}-\partial_t\varphi -\hat{v}_2\partial_2\varphi-\hat{v}_3\partial_3\varphi \\[3pt]
q-(\widehat{\mathcal{H}},\mathcal{H})+(\widehat{E},E)\\[3pt]
E_{\tau_2} -\varepsilon\widehat{\mathcal{H}}_3\partial_t\varphi -\varepsilon\varkappa\mathcal{H}_3+\widehat{E}_1\partial_2\varphi\\[3pt]
E_{\tau_3}+\varepsilon\widehat{\mathcal{H}}_2\partial_t\varphi +\varepsilon\varkappa\mathcal{H}_2+\widehat{E}_1\partial_3\varphi
\end{array}
\right)_{\bigl|}}_{x_1=0}=g,
\]
where $q=p+ (\widehat{H},H)$, $v_{N}= v_1-v_2\partial_2\widehat{\Psi}^{+}-v_3\partial_3\widehat{\Psi}^{+}$, $E_{\tau_j}=E_1\partial_j\widehat{\Psi}^{-}+E_j$, and the matrix
$\mathcal{C}(\widehat{U},\widehat{\Psi}^+)$ is determined as follows:
\[
\begin{array}{r}
\mathcal{C}(\widehat{U},\widehat{\Psi}^{+})Y
= (Y ,\nabla_yA_0(\widehat{U} ))\partial_t\widehat{U}
 +(Y ,\nabla_y\widetilde{A}_1(\widehat{U},\widehat{\Psi}^{+}))\partial_1\widehat{U}
 \\[6pt]
+ (Y ,\nabla_yA_2(\widehat{U} ))\partial_2\widehat{U}
+ (Y ,\nabla_yA_3(\widehat{U} ))\partial_3\widehat{U},
\end{array}
\]
\[
(Y ,\nabla_y A(\widehat{U})):=\sum_{i=1}^8y_i\left.\left(\frac{\partial A (Y )}{
\partial y_i}\right|_{Y =\widehat{U}}\right),\quad Y =(y_1,\ldots ,y_8).
\]

The differential operators $\mathbb{P}'(\widehat{U},\widehat{\Psi}^{+})$ and $\mathbb{V}'(\widehat{V},\widehat{\Psi}^{-})$ are first-order operators in $\Psi^{+}$ and $\Psi^{-}$ respectively. As in \cite{T12}, following Alinhac \cite{Al}, we introduce the ``good unknowns''
\begin{equation}
\dot{U}:=U -\frac{\Psi^+}{\partial_1\widehat{\Phi}^+}\,\partial_1\widehat{U},\qquad
\dot{V}:=V -\frac{\Psi^-}{\partial_1\widehat{\Phi}^-}\,\partial_1\widehat{V}.
\label{45}
\end{equation}
Omitting detailed calculations, we rewrite the linearized interior equations in terms of the new unknowns \eqref{45}:
\begin{eqnarray}
L(\widehat{U},\widehat{\Psi}^+)\dot{U} +\mathcal{C}(\widehat{U},\widehat{\Psi}^+)
\dot{U} + \frac{\Psi^+}{\partial_1\widehat{\Phi}^+}\,\partial_1\bigl\{\mathbb{P}
(\widehat{U},\widehat{\Psi}^+)\bigr\}=f_{\rm I}, \label{46} \\
M(\widehat{\Psi}^-)\dot{V} + \frac{\Psi^-}{\partial_1\widehat{\Phi}^-}\,\partial_1\bigl\{\mathbb{V}
(\widehat{V},\widehat{\Psi}^-)\bigr\}=f_{\rm II}.
\label{47}
\end{eqnarray}
Dropping as in \cite{Al,ST1,ST,T09,T09.cpam,T12} the zero-order terms in $\Psi^+$  and $\Psi^-$ in \eqref{46} and \eqref{47},\footnote{In the future nonlinear analysis the dropped terms in \eqref{46} and \eqref{47} should be considered as error terms at each Nash-Moser iteration step.} we write down the final form of our linearized problem for $(\dot{U},\dot{V},\varphi )$:
\begin{align}
 L(\widehat{U},\widehat{\Psi}^+)\dot{U} +\mathcal{C}(\widehat{U},\widehat{\Psi}^+)
\dot{U} =f_{\rm I}\qquad \mbox{in}\ \Omega_T,\label{48}\\
 M(\widehat{\Psi}^-)\dot{V}  =f_{\rm II}\qquad \mbox{in}\ \Omega_T,\label{49}
\end{align}
\begin{equation}
\left(
\begin{array}{c}
\dot{v}_{N}-\partial_t\varphi -\hat{v}_2\partial_2\varphi-\hat{v}_3\partial_3\varphi +
\varphi\,\partial_1\hat{v}_{N}\\[3pt]
\dot{q}-(\widehat{\mathcal{H}},\dot{\mathcal{H}})+(\widehat{E},\dot{E})+ [\partial_1\hat{q}]\varphi\\[3pt]
\dot{E}_{\tau_2}-\varepsilon
\partial_t(\widehat{\mathcal{H}}_3\varphi ) -\varepsilon\varkappa\dot{\mathcal{H}}_3+\partial_2(\widehat{E}_1\varphi )\\[3pt]
\dot{E}_{\tau_3}+\varepsilon\partial_t(\widehat{\mathcal{H}}_2\varphi ) +\varepsilon\varkappa\dot{\mathcal{H}}_2+\partial_3(\widehat{E}_1\varphi )
\end{array}
\right)=g \qquad \mbox{on}\ \partial\Omega_T,
\label{50}
\end{equation}
\begin{equation}
(\dot{U},\dot{V},\varphi )=0\qquad \mbox{for}\ t<0,\label{51}
\end{equation}
where $[\partial_1\hat{q}]=(\partial_1\hat{q})|_{x_1=0}-(\widehat{\mathcal{H}},\partial_1\widehat{\mathcal{H}})|_{x_1=0}+(\widehat{E},\partial_1\widehat{E})|_{x_1=0}$ and all of the values with dots ($\dot{v}_{N}$, $\dot{q}$, etc.) are determined similarly to corresponding values without dots. We used  \eqref{42} taken at $x_1=0$ while writing down the last two boundary conditions in \eqref{50}. We assume that $f$ and $g$ vanish in the past and consider the case of zero initial data, which is the usual assumption.\footnote{The case of nonzero initial data is postponed to the nonlinear analysis (construction of a so-called approximate solution; see, e.g., \cite{T09}).}

\subsection{Function spaces}

Thanks to assumption \eqref{expan.lin} the boundary matrix $\widehat{\mathfrak{B}}:=\mathfrak{B} (\hat{\varphi} )$ (see \eqref{bm}) for system \eqref{49}  is non-singular. At the same time, in view of the first conditions in \eqref{40} and \eqref{43}, the boundary matrix $\widehat{\mathfrak{A}}:=\mathfrak{A}(\widehat{U}_{|x_1=0},\hat{\varphi})$ for system \eqref{48} is singular, i.e., the boundary $x_1=0$ is characteristic for the linearized MHD equations \eqref{48}. Indeed, let
\[
W=(\dot{q},\dot{v}_N,\dot{v}_2,\dot{v}_3,\dot{H}_N,\dot{H}_2,\dot{H}_3,\dot{S})
\]
and $\dot{U}=JW$. Clearly, $\det J\neq 0$. Using the first conditions in \eqref{40} and \eqref{43}, after some algebra we obtain
\begin{equation}
(\widehat{\mathfrak{A}}\dot{U},\dot{U})=2\dot{q}\dot{v}_N.
\label{bmq}
\end{equation}
It follows from \eqref{bmq} that
\begin{equation}
(\widehat{\mathfrak{A}}\dot{U},\dot{U})= (J^{T}\widehat{\mathfrak{A}}J\,W,W)=
(\mathcal{E}_{12}W,W),\qquad \mathcal{E}_{12}=\left(\begin{array}{ccccc}
0& 1 &0 & \cdots & 0 \\
1 & 0 &0 & \cdots & 0 \\
0 & 0 &0 & \cdots & 0 \\
\vdots & \vdots &\vdots & & \vdots \\
0& 0 &0 & \cdots & 0  \end{array}
 \right).
\label{35'}
\end{equation}
Hence, the matrix $\widehat{\mathfrak{A}}$ has one positive and one negative eigenvalue, and other eigenvalues are zeros.

The fact that the boundary $x_1=0$ is characteristic for the MHD system implies a natural loss of control on derivatives in the normal direction. It is known that in MHD, unlike the situation in gas dynamics (see, e.g., \cite{T09.cpam}), this loss of control on derivatives cannot be compensated and the natural functional setting is provided by the anisotropic weighted Sobolev spaces $H^m_*$.

The functional space $H^m_*$ is defined as follows (see \cite{Chen,YM,MST,Sec}):
$$
H^{m}_*(\mathbb{R}^3_+ ):=\left\{ u\in L_2(\mathbb{R}^3_+) \ | \
\partial^{\alpha}_*\partial_1^k u\in L_2(\mathbb{R}^3_+ )\quad \mbox{if}\quad |\alpha
|+2k\leq m\, \right\},
$$
where $m\in\mathbb{N}$, $\partial^{\alpha}_*=(\sigma \partial_1)^{\alpha_1}\partial_2^{\alpha_2}
\partial_3^{\alpha_3}\,$, and  $\sigma (x_1)\in C^{\infty}(\mathbb{R}_+)$ is
a monotone increasing function such that $\sigma (x_1)=x_1$ in a neighborhood of
the origin and $\sigma (x_1)=1$ for $x_1$ large enough. The space $H^m_*(\mathbb{R}^3_+ )$ is normed by
\[
\|u\|_{m,*}^2= \sum_{|\alpha |+2k\leq m}
\|\partial^{\alpha}_*\partial_1^k u\|_{L_2(\mathbb{R}^3_+)}^2.
\]
We also define the space
\[
H^m_*(\Omega_T)=\bigcap_{k=0}^{m}H^k((-\infty ,T],H^{m-k}_*(\mathbb{R}^3_+ ))
\]
equipped with the norm
\[
[u]^2_{m,*,T}=\int_{-\infty}^T\nt u (t)\nt^2_{m,*}dt,
\quad \mbox{where}\quad
\nt u (t)\nt^2_{m,*}=\sum\limits_{j=0}^m\|\partial_t^ju(t)\|^2_{m-j,*}.
\]
Within this paper we use the space $H^m_*(\Omega_T)$ mainly for $m=1$. Clearly,
the norm for $H^1_*(\Omega_T)$ reads
\[
[u]^2_{1,*,T}=\int_{\Omega_T}\left( u^2 +(\partial_tu)^2 +(\sigma\partial_1u)^2
+(\partial_2u)^2+(\partial_3u)^2\right) dtdx.
\]

\subsection{Reduced linearized problem with homogeneous Maxwell equations and \protect\\ boundary conditions}

Following (with technical modifications) arguments in \cite{T12} we can pass from the unknown  $(\dot{U},\dot{V})$ to such a new unknown $(U^{\natural},V^{\natural})$ that it satisfies problem \eqref{48}--\eqref{51} with $f_{\rm II}=0$, $g=0$ and $f_{\rm I}$ replaced by a vector-function $F = (F_1,\ldots ,F_4, 0,0,0,F_8)$ obeying the estimate
\begin{equation}
[F]_{1,*,T}
 \leq C\left\{ [f_{\rm I}]_{3,*,T}+\|f_{\rm II}\|_{H^{3}(\Omega_T)}+\|g\|_{H^{3}(\partial\Omega_T)}\right\}.
 \label{88}
\end{equation}
Here and later on $C$ is a constant that can change from line to line, and it may depend from another constants. In particular, in \eqref{88} the constant $C$ depends on $K$ and $T$. Dropping for convenience the indices $^{\natural}$, we write down  our reduced linearized problem for the new unknown $({U},{V}):=(U^{\natural},V^{\natural})$ and the interface perturbation $\varphi$:
\begin{align}
& \widehat{A}_0\partial_t{U}+\sum_{j=1}^{3}\widehat{A}_j\partial_j{U}+
\widehat{\mathcal{C}}{U}=F \qquad \mbox{in}\ \Omega_T,\label{89}
\\
& \varepsilon\partial_t{V}+\widehat{B}_1\partial_1{V}+{B}_2\partial_2{V}+
{B}_3\partial_3{V}=0 \qquad \mbox{in}\ \Omega_T,\label{90}
\end{align}
\begin{subequations}
\begin{align}
& \partial_t\varphi={v}_{N}-\hat{v}_2\partial_2\varphi-\hat{v}_3\partial_3\varphi +
\varphi\,\partial_1\hat{v}_{N}, \label{91a} \\[3pt]
& {q}=(\widehat{\mathcal{H}},{\mathcal{H}})-(\widehat{E},{E})-  [ \partial_1\hat{q}] \varphi  ,\label{91b} \\[3pt]
& {E}_{\tau_2}=
\varepsilon\partial_t(\widehat{\mathcal{H}}_3\varphi )
-\partial_2(\widehat{E}_1\varphi )+\varepsilon\varkappa{\mathcal{H}}_3,\label{91c}  \\[3pt]
& {E}_{\tau_3}=-\varepsilon\partial_t(\widehat{\mathcal{H}}_2\varphi )
-\partial_3(\widehat{E}_1\varphi ) -\varepsilon\varkappa{\mathcal{H}}_2\qquad  \mbox{on}\ \partial\Omega_T,\label{91d}
\end{align}\label{91}
\end{subequations}
\begin{equation}
({U},{V},\varphi )=0\qquad \mbox{for}\ t<0,\label{92}
\end{equation}
where $\widehat{A}_{\alpha}:={A}_{\alpha}(\widehat{U})$ $(\alpha =0,2,3)$,
$\widehat{A}_1:=\widetilde{A}_1(\widehat{U},\widehat{\Psi}^+)$,
$\widehat{\mathcal{C}}:=\mathcal{C}(\widehat{U},\widehat{\Psi}^+)$, and $\widehat{B}_1:=\widetilde{B}_1(\widehat{\Psi}^-)$.

Since the boundary conditions \eqref{91}, equations \eqref{90} and the fifth, sixth and seventh equations in \eqref{89} are homogeneous, following arguments in \cite{T12} (with technical modifications necessary for our non-relativistic settings), we can prove that solutions to problem \eqref{89}--\eqref{92} satisfy
\begin{equation}
{\rm div}\,{h}=0\qquad\mbox{in}\ \Omega_T,
\label{93}
\end{equation}
\begin{equation}
{\rm div}\,{\mathfrak{h}}=0,\quad {\rm div}\,{\mathfrak{e}}=0\qquad\mbox{in}\ \Omega_T,
\label{94}
\end{equation}
\begin{equation}
{H}_{N}=\widehat{H}_2\partial_2\varphi +\widehat{H}_3\partial_3\varphi -
\varphi\,\partial_1\widehat{H}_{N}\qquad\mbox{on}\ \partial\Omega_T,
\label{95}
\end{equation}
\begin{equation}
{\mathcal{H}}_{N}=\widehat{\mathcal{H}}_2\partial_2\varphi +\widehat{\mathcal{H}}_3\partial_3\varphi -
\varphi\,\partial_1\widehat{\mathcal{H}}_{N}\qquad\mbox{on}\ \partial\Omega_T,
\label{96}
\end{equation}
where
\[
h=(H_N,H_2\partial_1\widehat{\Phi}^+,H_3\partial_1\widehat{\Phi}^+),\quad
H_N=H_1-H_2\partial_2\widehat{\Psi}^+-H_3\partial_3\widehat{\Psi}^+,
\]
\[
\mathfrak{h}=(\mathcal{H}_{N},\mathcal{H}_2\partial_1\widehat{\Phi}^{-},\mathcal{H}_3\partial_1\widehat{\Phi}^{-}),\quad
\mathcal{H}_{N}=\mathcal{H}_1-\mathcal{H}_2\partial_2\widehat{\Psi}^{-}-\mathcal{H}_3\partial_3\widehat{\Psi}^{-},
\]
\[
\mathfrak{e}=({E}_{N},{E}_2\partial_1\widehat{\Phi}^{-},{E}_3\partial_1\widehat{\Phi}^{-}),\quad {E}_{N}={E}_1-{E}_2\partial_2\widehat{\Psi}^{-}-{E}_3\partial_3\widehat{\Psi}^{-}.
\]
Moreover, again referring to \cite{T12} for detailed arguments, we can estimate solutions of problem \eqref{48}--\eqref{51} through solutions of problem \eqref{89}--\eqref{92}:
\begin{multline}
[\dot{U}]_{1,*,T}+\|\dot{V}\|_{H^{1}(\Omega_T)}
\leq [{U}]_{1,*,T}+\|{V}\|_{H^{1}(\Omega_T)} \\+C\left\{ [f_{\rm I}]_{3,*,T}+\|f_{\rm II}\|_{H^{3}(\Omega_T)}+
\|g\|_{H^{3}(\partial\Omega_T)}\right\}.
\label{88a}
\end{multline}

Taking into account estimate \eqref{88a}, from now on we concentrate on the study of the reduced linearized problem \eqref{89}--\eqref{92} keeping in mind that its well-posedness under suitable assumptions on the basic state and the regularity of the data $F$ implies the well-posedness of problem \eqref{48}--\eqref{51} for which the regularity of the data should be consistent with estimate \eqref{88}.

\subsection{Constant coefficients linearized problem for a planar interface}

If we ``freeze'' coefficients of problem \eqref{89}--\eqref{92}, drop zero-order terms, assume that $\partial_2\hat{\varphi}=\partial_3\hat{\varphi}=0$, and in the change of variables take $\chi\equiv 1$, then we obtain a constant coefficients linear problem which is the result of the linearization of the original nonlinear free boundary value problem
\eqref{4}, \eqref{13}, \eqref{16}, \eqref{19}, \eqref{20} about its {\it exact} constant solution
\begin{align}
& U=\widehat{U}=(\hat{p},\hat{v},\widehat{H},\widehat{S})={\rm const}\qquad \mbox{for}\ x_1>\varkappa t,\nonumber \\
& V=\widehat{V}=(\widehat{\mathcal{H}},\widehat{E})={\rm const}\qquad \mbox{for}\ x_1<\varkappa t\nonumber
\end{align}
for the planar plasma-vacuum interface $x_1 =\varkappa t$, where $\varkappa$ is a constant interface speed. This exact constant solution satisfies \eqref{40} and \eqref{43}:
\begin{equation}
\hat{v}_1=\varkappa ,\quad \widehat{H}_1=\widehat{\mathcal{H}}_1=0,\quad \widehat{E}_2=\varepsilon\varkappa\widehat{\mathcal{H}}_3,\quad
\widehat{E}_3=-\varepsilon\varkappa\widehat{\mathcal{H}}_2.
\label{98}
\end{equation}
Moreover, since the original nonlinear equations were already written in a dimensionless form (see \eqref{12}), without loss of generality we can suppose that
\begin{equation}
\hat{\rho}=1\quad\mbox{and}\quad\hat{a}=1
\label{98'}
\end{equation}
(then $M=\sqrt{\hat{v}_2^2+\hat{v}_3^2}$ is the Mach number).

Taking into account \eqref{98} and \eqref{98'}, we have the following constant coefficients problem:
\begin{align}
& \partial_t{U}+\sum_{j=1}^{3}\widehat{A}_j\partial_j{U}=F \qquad \mbox{in}\ \Omega_T,\label{99}
\\
& \varepsilon\partial_t{V}+\widehat{B}_1\partial_1{V}+{B}_2\partial_2{V}+
{B}_3\partial_3{V}=0 \qquad \mbox{in}\ \Omega_T,\label{100}
\end{align}
\begin{equation}
\left\{ \begin{array}{ll}
\partial_t\varphi={v}_1-\hat{v}_2\partial_2\varphi-\hat{v}_3\partial_3\varphi , & \\[6pt]
{q}=\widehat{\mathcal{H}}_2({\mathcal{H}}_2+\varepsilon\varkappa E_3) +\widehat{\mathcal{H}}_3({\mathcal{H}}_3-\varepsilon\varkappa E_2)-\widehat{E}_1{E}_1  , & \\[6pt]
{E}_{2}=\varepsilon\widehat{\mathcal{H}}_3\partial_t \varphi
-\widehat{E}_1\partial_2 \varphi +\varepsilon\varkappa{\mathcal{H}}_3, & \\[6pt]
{E}_{3}=-\varepsilon\widehat{\mathcal{H}}_2\partial_t\varphi
-\widehat{E}_1\partial_3 \varphi  -\varepsilon\varkappa{\mathcal{H}}_2,& \quad \mbox{on}\ \partial\Omega_T,
\end{array}\right.\label{101}
\end{equation}
\begin{equation}
({U},{V},\varphi )=0\qquad \mbox{for}\ t<0,\label{102}
\end{equation}
and solutions to problem \eqref{99}--\eqref{102} satisfy
\begin{equation}
{\rm div}\,H=0,\quad {\rm div}^-\,{\mathcal{H}}=0,\quad {\rm div}^-\,E=0\qquad\mbox{in}\ \Omega_T,
\label{103}
\end{equation}
\begin{equation}
{H}_{1}=\widehat{H}_2\partial_2\varphi +\widehat{H}_3\partial_3\varphi ,\quad
{\mathcal{H}}_{1}=\widehat{\mathcal{H}}_2\partial_2\varphi +\widehat{\mathcal{H}}_3\partial_3\varphi \qquad\mbox{on}\ \partial\Omega_T,
\label{104}
\end{equation}
where ${\rm div}^-\,a=-\partial_1a_1+\partial_2a_2+\partial_3a_3$ for any vector $a=(a_1,a_2,a_3)$ and
\[
\widehat{A}_1=\left( \begin{array}{cccccccc} 0&1&0&0&0&0&0&0\\[6pt]
1&0&0&0&0&\widehat{H}_2&\widehat{H}_3&0\\
0&0&0&0&0&0&0&0\\ 0&0&0&0&0&0&0&0\\
0&0&0&0&0&0&0&0\\
0&\widehat{H}_2&0&0&0&0&0&0\\
0&\widehat{H}_3&0&0&0&0&0&0\\ 0&0&0&0&0&0&0&0\\
\end{array} \right) ,\quad
\widehat{B}_1=\left(
\begin{array}{cccccc}
\varepsilon\varkappa & 0 & 0& 0 & 0 & 0 \\
0 & \varepsilon\varkappa & 0& 0 & 0 & 1 \\
0 & 0 & \varepsilon\varkappa& 0 & -1 & 0 \\
0 & 0 & 0& \varepsilon\varkappa & 0 & 0 \\
0 & 0 & -1& 0 & \varepsilon\varkappa & 0 \\
0 & 1 & 0& 0 & 0 & \varepsilon\varkappa
\end{array}
\right),
\]
\[
\widehat{A}_2=\left( \begin{array}{cccccccc} {\hat{v}_2}&0&1&0&0&0&0&0\\[6pt]
0&\hat{v}_2&0&0&-{\widehat{H}_2}&0&0&0\\ 1&0&\hat{v}_2&0&0&0&{\widehat{H}_3}&0\\ 0&0&0&
\hat{v}_2&0&0&-{\widehat{H}_2}&0\\
0&-{\widehat{H}_2}&0&0&{\hat{v}_2}&0&0&0\\
0&0&0&0&0&{\hat{v}_2}&0&0\\
0&0&{\widehat{H}_3}&-{\widehat{H}_2}&0&0&{\hat{v}_2}&0\\0&0&0&0&0&0&0&\hat{v}_2
\end{array} \right) ,
\]
\[
\widehat{A}_3=\left( \begin{array}{cccccccc} {\hat{v}_3}&0&0&1&0&0&0&0\\[6pt]
 0&\hat{v}_3&0&0&-{\widehat{H}_3}&0&0&0\\ 0&0&\hat{v}_3&0&0&-{\widehat{H}_3}&0&0\\ 1&0&0&
\hat{v}_3&0&{\widehat{H}_2}&0&0\\
0&-{\widehat{H}_3}&0&0&{\hat{v}_3}&0&0&0\\
0&0&-{\widehat{H}_3}&{\widehat{H}_2}&0&{\hat{v}_3}&0&0\\
0&0&0&0&0&0&{\hat{v}_3}&0\\ 0&0&0&0&0&0&0&\hat{v}_3
\end{array} \right).
\]

From the physical point of view, the well-posedness of problem \eqref{99}--\eqref{102} can be interpreted as the stability of a planar relativistic plasma-vacuum interface (or the  macroscopic stability of a corresponding nonplanar interface).

\subsection{Main results}

We are now in a position to state the main results of this paper.

\begin{theorem}
Let the basic state \eqref{37} satisfies assumptions \eqref{38}--\eqref{44}. Let also
\begin{equation}
{|\widehat{H}_2\widehat{\mathcal{H}}_3-\widehat{H}_3\widehat{\mathcal{H}}_2|
_{|}}_{x_1=0}
\geq \varepsilon_2> 0,\label{52}
\end{equation}
where $\varepsilon_2$ is a fixed constant. Then there exists a positive constant $\widehat{E}_1^*$ such that if the basic state satisfies the condition $|\widehat{E}_{1}|<\widehat{E}_1^*$ on $\partial\Omega_T$, then  problem \eqref{89}--\eqref{92} has a unique solution $(U,V,\varphi)\in H^1_*(\Omega_T)\times H^1(\Omega_T)\times H^{3/2}(\partial\Omega_T)$ for all $F \in H^1_*(\Omega_T)$ which vanish in the past. Moreover, the solution obeys
the a priori estimate
\begin{equation}
[{U}]_{1,*,T}+\|{V}\|_{H^{1}(\Omega_T)}+\|\varphi\|_{H^{3/2}(\partial\Omega_T)}\leq C[F]_{1,*,T},
\label{54}
\end{equation}
where $C=C(K,T,\varepsilon_2 )>0$ is a constant independent of the data $F$.
\label{t1}
\end{theorem}

\begin{remark}
{\rm In fact, we prove Theorem \ref{t1} for $\varkappa \leq 0$, i.e., including the case when the interface speed is zero somewhere (or even everywhere) on the unperturbed interface. But, we prefer to keep assumption \eqref{expan.lin} in the formulation of Theorem \ref{t1} for its future usage in the nonlinear analysis of the plasma-vacuum problem.
}
\label{r2}
\end{remark}

We note that the existence of solutions of the linearized problem was not proved in \cite{T12}, but our arguments towards the proof of existence in this paper are directly applicable in relativistic settings in \cite{T12} as well. In other words, we have the side result that is the well-posedness of the linearized problem for relativistic plasma-vacuum interfaces under a suitable stability condition from \cite{T12}.

In view of \eqref{40},
\[
|\widehat{E}_{1}| =\frac{|\widehat{E}|}{\sqrt{1+(\partial_2\hat{\varphi})^2+(\partial_3\hat{\varphi})^2}}+\mathcal{O}(\varepsilon )\qquad\mbox{on}\ \partial\Omega_T.
\]
That is, Theorem \ref{t1} says that if the unperturbed vacuum electric field on the interface is small enough, then under the non-collinearity condition \eqref{52} the linearized problem is well-posed. The natural question is whether a large enough vacuum electric field can make the linearized problem ill-posed. As in relativistic settings in \cite{T12}, the answer on this question is positive, but in our non-relativistic case we can analyze the influence of the vacuum electric field on well-posedness in more detail (at least, for particular cases of the unperturbed flow).

Clearly, the ill-posedness of the corresponding ``frozen'' coefficients linearized problem indicates the ill-posedness of the problem with variable coefficients. In other words, we may study the influence of the constant vacuum electric field on the well-posedness of the constant coefficients linearized problem \eqref{99}--\eqref{102}. Its ill-posedness will mean the violent (Kelvin-Helmholz) instability of a planar interface. Mathematically this means that there is a range of admissible parameters of problem \eqref{99}--\eqref{102} for which the Kreiss-Lopatinski condition \cite{Kreiss,Maj84} is violated.

The test of the Kreiss-Lopatinski condition is equivalent to the usual normal modes analysis (spectral analysis), i.e., the construction of an Hadamard-type ill-posedness example. Unfortunately, due to principal technical difficulties a complete normal modes analysis of problem \eqref{99}--\eqref{101} seems impossible (even numerically). In this paper we restrict ourselves to the following particular case of the uniform unperturbed flow for which a part of the analysis can be done analytically:
\begin{equation}
\hat{v}_1=\hat{v}_2=0,\quad \widehat{H}_2=\widehat{\mathcal{H}}_3=0,\quad \widehat{H}_3\widehat{\mathcal{H}}_2\neq 0.
\label{pcase}
\end{equation}
Strictly speaking, since $\hat{v}_1=\varkappa$ (see \eqref{98}), we have $\varkappa =0$. However, by continuity argument this can be extended to the case $|\varkappa | \ll 1$.

The (rough) sufficient stability condition $|\widehat{E}_1|\ll 1$ for a planar interface can be {\it essentially specified} by a numerical analysis of the positive definiteness of some matrix of order 42 (see the next section). But, for the particular case \eqref{pcase} this can be even done analytically. This enables us to compare our (specified) sufficient stability condition with the spectral stability condition describing the whole stability domain. We summarize our results for the particular case \eqref{pcase} in the following proposition.

\begin{proposition}
Let the uniform (piecewise constant) flow with the planar plasma-vacuum interface $x_1=0$ satisfies conditions \eqref{pcase}. If the flow parameters satisfy the inequalities
\begin{equation}
\left\{
\begin{array}{l}
\widehat{E}_{1}^2<\ds\frac{1}{2}\wh{\vH}{}_2^2,\\[9pt]
\widehat{E}_{1}^2 \bigg(1+\ds\frac{\hat{v}_3^2}{\wh{H}{}_3^2} \bigg)<1,\\[12pt]
2 \widehat{E}_{1}^4 \bigg( \ds\frac{\wh{\vH}{}_2^2 + \wh{H}{}_3^2 +
\hat{v}_3^2}{\wh{H}{}_3^2 \wh{\vH}{}_2^2}\bigg) - \widehat{E}_{1}^2 \bigg( 1+
2 \ds\frac{\wh{H}{}_3^2+\wh{\vH}{}_2^2}{\wh{H}{}_3^2
\wh{\vH}{}_2^2}+\ds\frac{\hat{v}_3^2}{\wh{H}{}_3^2} \bigg)+1>0,\\[12pt]
2 \widehat{E}_{1}^4 \bigg( \ds\frac{3+\hat{v}_3^2}{\wh{H}{}_3^2} \bigg) -
\widehat{E}_{1}^2 \bigg( 3+\hat{v}_3^2+2
\ds\frac{1+\hat{v}_3^2}{\wh{H}{}_3^2}  \bigg) +1>0,
\end{array}\right.\label{posdef}
\end{equation}
then the planar interface is linearly stable and solutions to problem \eqref{99}--\eqref{102} obey the energy estimate \eqref{54}. For the static case $\hat{v}_3=0$ inequalities \eqref{posdef} are equivalent to the condition
\begin{equation}
\widehat{E}_{1}^2<\min\left\{ \frac{1}{3},\,\frac{\widehat{\mathcal{H}}_2^2\widehat{H}_3^2}{2(\widehat{\mathcal{H}}_2^2+\widehat{H}_3^2)}\right\}
\label{v3=0}
\end{equation}
and if
\begin{equation}
\widehat{E}_{1}^2>\widehat{\mathcal{H}}_2^2,
\label{anins}
\end{equation}
then the planar interface is violently unstable. Moreover, condition \eqref{anins} is sufficient for instability whereas the whole instability domain is wider than that described by \eqref{anins} and can be found numerically. The whole instability domain (for the static case $\hat{v}_3=0$) is represented in Fig. 1 for certain fixed values of $\widehat{H}_3$.
\label{ppcase}
\end{proposition}

\begin{figure}[t]
\centering
\begin{minipage}[h]{0.46\linewidth}
    \center{\includegraphics[width=\linewidth]{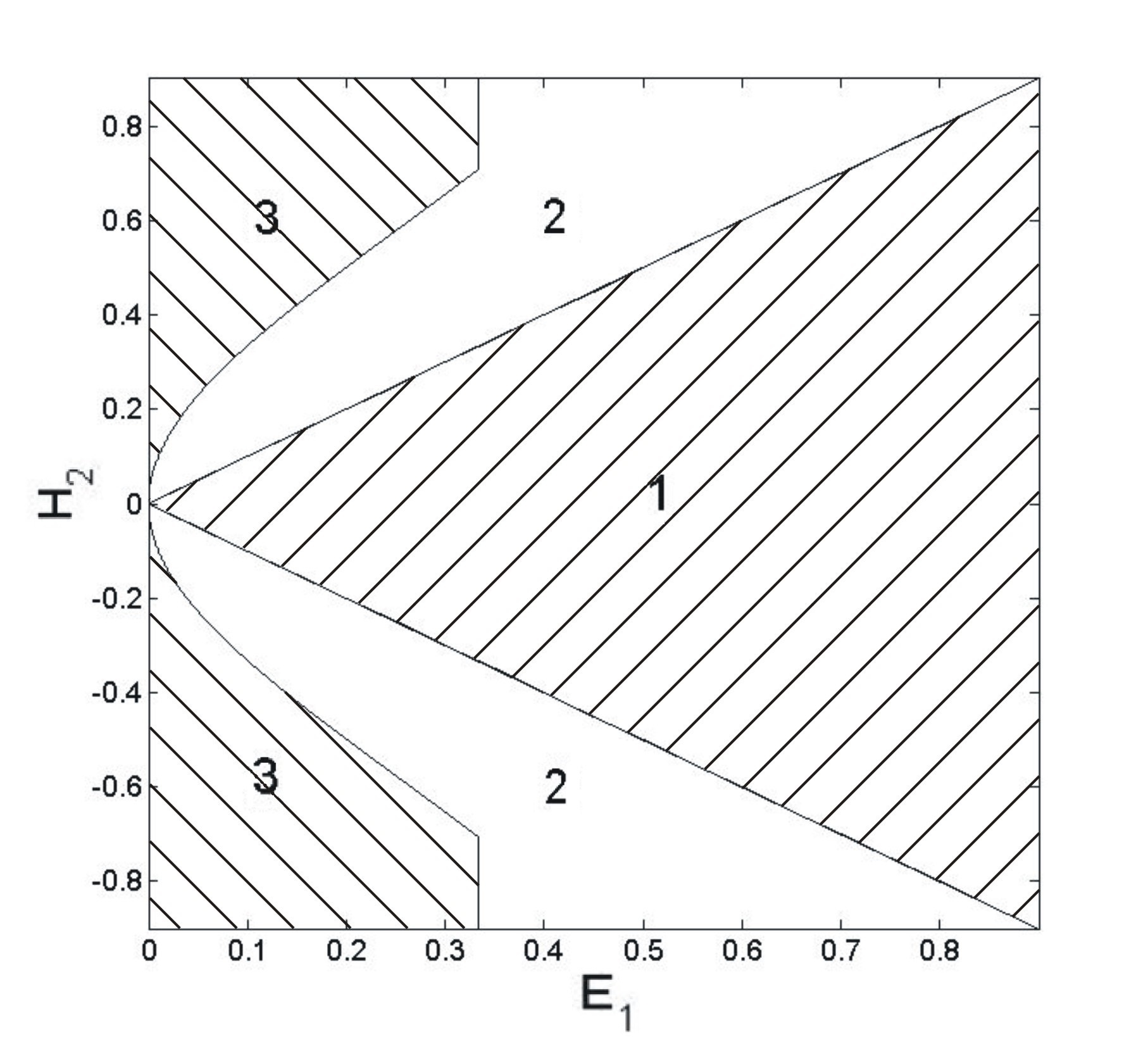} \\[-7pt]a) $\wh{H}_3 = 1$}
\end{minipage}
\begin{minipage}[h]{0.46\linewidth}
    \center{\includegraphics[width=\linewidth]{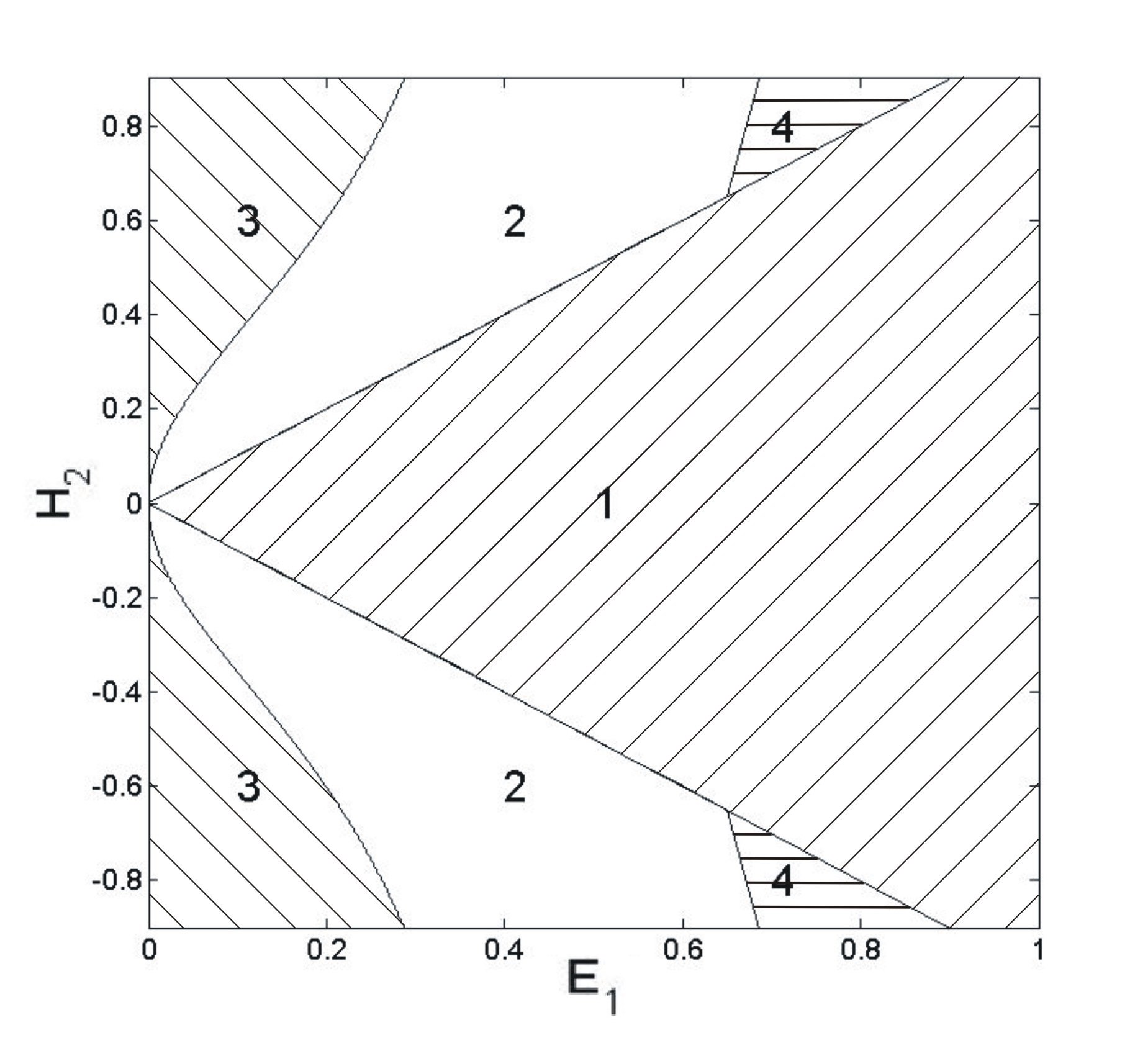} \\[-7pt]b) $\wh{H}_3 = 2/3$}
\end{minipage}
\vfill
\begin{minipage}[h]{0.46\linewidth}
    \center{\includegraphics[width=\linewidth]{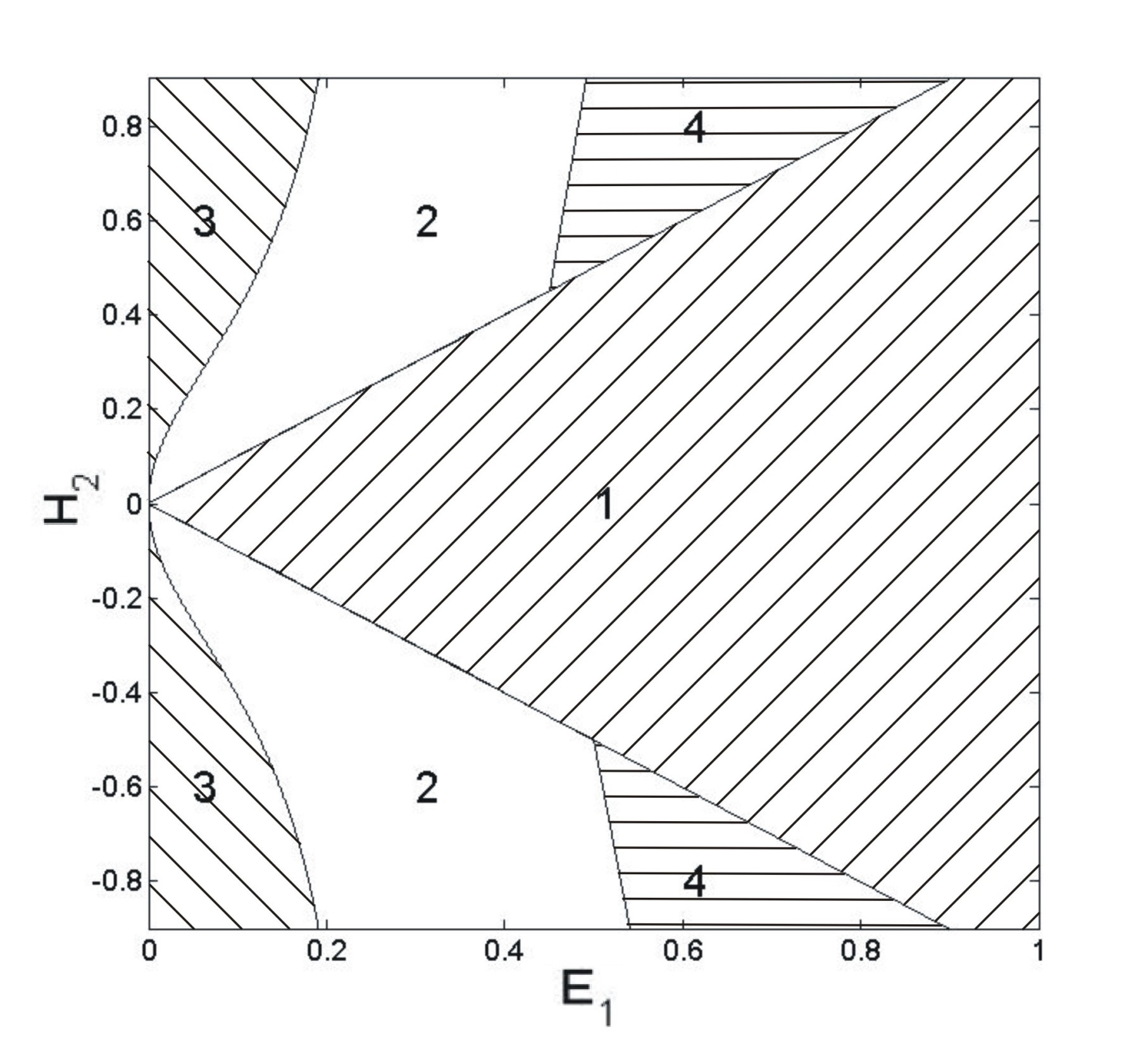} \\[-7pt]c) $\wh{H}_3 = 0.5$}
\end{minipage}
\begin{minipage}[h]{0.46\linewidth}
    \center{\includegraphics[width=\linewidth]{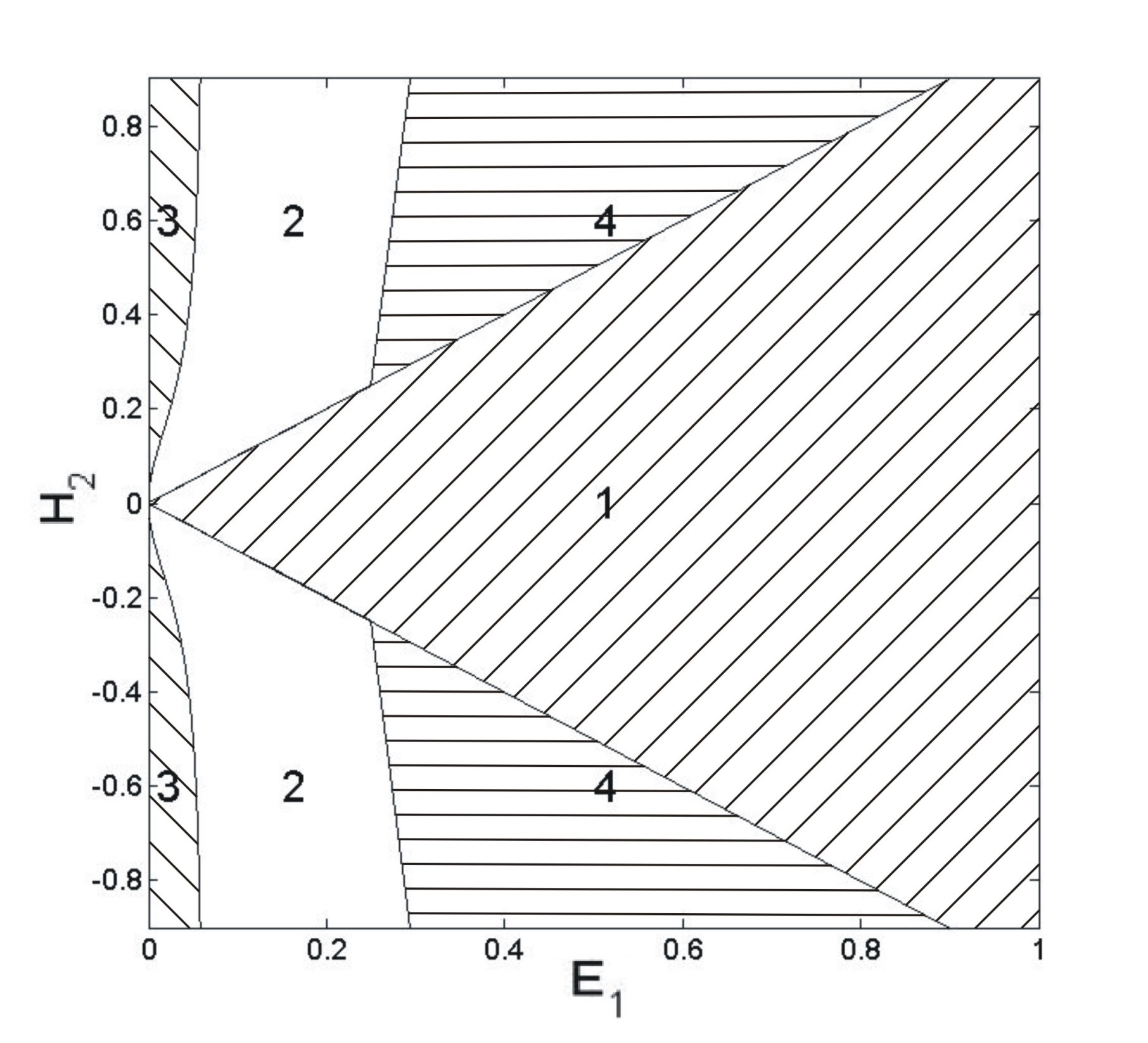} \\[-7pt]d) $\wh{H}_3 = 0.25$}
\end{minipage}
\caption{Domains 1 and 3 are the domains described by the sufficient instability condition \eqref{anins} and the sufficient stability condition \eqref{v3=0} respectively. The union of domains 2 and 3 is the whole domain of stability, and the union of domains 1 and 4 is the  whole domain of instability.}
\end{figure}

As we can see in Fig.1, the domain described by the sufficient stability condition \eqref{v3=0} (domain 3) is relatively big in comparison with the whole stability domain (the union of domains 2 and 3). It is important that the sufficient stability condition \eqref{v3=0} and its non-static counterpart \eqref {posdef} are found analytically for the particular case \eqref{pcase}. But, even in the general case of the unperturbed flow a numerical analysis of the positive definiteness  of some matrix of order 42  appearing in our energy method (see the next section) seems straightforward. At the same time, we are able to find the whole stability domain by a numerical realization of the normal modes analysis only for relatively simple particular cases like \eqref{pcase} whereas in the general case this is connected with very expensive numerical calculations. Moreover, the number of dimensionless parameters of the unperturbed flow is really big, and a complete numerical test of the Kreiss-Lopatinski condition seems unrealizable in practice.

Thus, the main goal of our calculations for the particular case \eqref{pcase} was to show that the unshaded interlayer in Fig.1 (domain 2) between the instability domain (the union of domains 1 and 4) and the part of the whole stability domain found by the energy method is not, in a certain (loose) sense, extremely big (in comparison with domain 3). Moreover, our energy method enables one to derive the a priori estimate \eqref{54}. To do the same by starting from the normal modes analysis we should  construct a Kreiss-type symmetrizer \cite{Kreiss,Maj84} that is really difficult and often technically impossible in MHD (especially, for the case of characteristic boundary).

\section{Energy method for the constant coefficients problem}
\label{sec:4}

For the derivation of the a priori estimate \eqref{54} we should make technical modifications in the relativistic case in \cite{T12} where a dissipative energy integral for a prolonged linearized system was constructed. These modifications are connected only with unimportant peculiarities of our problem in comparison with its relativistic counterpart. We could just try to explain these modifications by referring to the arguments in \cite{T12}. However, this would be, first, very inconvenient for the reader and, second, the main reason that we prefer to do not drop the process of derivation of estimate \eqref{54} is that in Section \ref{sec:6} we will need details of the construction of the dissipative energy integral for the proof of existence of solutions. At the same time, here, unlike \cite{T12}, we prove the  a priori estimate only for  the case of constant coefficients. In Section \ref{sec:6} we mainly concentrate on the proof of existence of solutions and drop a detailed description of the extension of the a priori estimate to the case of variable coefficients because corresponding arguments are really similar to those in \cite{T12}.

For problem \eqref{99}--\eqref{102} we first derive the a priori estimate \eqref{54} under the (rough) sufficient stability condition that the constant
$|\widehat{E}_1|$ is small enough and the vectors $(\widehat{H}_2,\widehat{H}_3)$ and $(\widehat{\mathcal{H}}_2,\widehat{\mathcal{H}}_3)$ are nonzero and nonparallel to each other (cf. \eqref{52}), i.e.,
\begin{equation}
\widehat{H}_2\widehat{\mathcal{H}}_3-\widehat{H}_3\widehat{\mathcal{H}}_2\neq 0.\label{52a}
\end{equation}
After that on the example of the particular case \eqref{pcase} we show that the rough condition $|\widehat{E}_1|\ll 1$ can be essentially specified and for this case we get the sufficient stability condition \eqref{posdef}.

As in \cite{T12}, the crucial role in deriving the a priori estimate for the linearized problem is played by a secondary symmetrization of the Maxwell equations \eqref{13}. Following \cite{T12}, we equivalently rewrite the symmetric system \eqref{13},
\[
\varepsilon\partial_tV +\sum_{j=1}^3B_j\partial_jV=0,
\]
as the new symmetric system
\begin{equation}
\varepsilon\mathcal{B}_0\partial_tV +\sum_{j=1}^3\mathcal{B}_j\partial_jV=0,
\label{20'}
\end{equation}
provided that the hyperbolicity condition $\mathcal{B}_0>0$ holds, where
\[
\mathcal{B}_0=\left(\begin{array}{cccccc}
1 & 0 & 0& 0 & \nu_3 & -\nu_2 \\
0 & 1 & 0& -\nu_3 & 0 & \nu_1 \\
0 & 0 & 1& \nu_2 & -\nu_1 & 0 \\
0 & -\nu_3 & \nu_2& 1 & 0 & 0 \\
\nu_3 & 0 & -\nu_1& 0 & 1 & 0 \\
-\nu_2 & \nu_1 & 0& 0 & 0 & 1
\end{array} \right),\quad
\mathcal{B}_1=
\left(\begin{array}{cccccc}
\nu_1 & \nu_2 & \nu_3& 0 & 0 & 0 \\
\nu_2 & -\nu_1 & 0& 0 & 0 & -1 \\
\nu_3 & 0 & -\nu_1& 0 & 1 & 0 \\
0 & 0 & 0& \nu_1 & \nu_2 & \nu_3 \\
0 & 0 & 1& \nu_2 & -\nu_1 & 0 \\
0 & -1 & 0& \nu_3 & 0 & -\nu_1
\end{array} \right),
\]
\[
\mathcal{B}_2=
\left(\begin{array}{cccccc}
-\nu_2 & \nu_1 & 0& 0 & 0 & 1 \\
\nu_1 & \nu_2 & \nu_3& 0 & 0 & 0 \\
0 & \nu_3 & -\nu_2& -1 & 0 & 0 \\
0 & 0 & -1& -\nu_2 & \nu_1 & 0 \\
0 & 0 & 0& \nu_1 & \nu_2 & \nu_3 \\
1 & 0 & 0& 0 & \nu_3 & -\nu_2
\end{array} \right),\quad
\mathcal{B}_3=
\left(\begin{array}{cccccc}
-\nu_3 & 0 & \nu_1& 0 & -1 & 0 \\
0 & -\nu_3 & \nu_2& 1 & 0 & 0 \\
\nu_1 & \nu_2 & \nu_3& 0 & 0 & 0 \\
0 & 1 & 0& -\nu_3 & 0 & \nu_1 \\
-1 & 0 & 0& 0 & -\nu_3 & \nu_2 \\
0 & 0 & 0& \nu_1 & \nu_2 & \nu_3
\end{array} \right),
\]
and  $\nu_i$ are arbitrary functions $\nu_i(t,x)$ satisfying the condition $\mathcal{B}_0>0$, i.e.,
\begin{equation}
|\nu |<1,\label{21'}
\end{equation}
with the vector-function $\nu = (\nu_1, \nu_2, \nu_3)$.

Using the secondary symmetrization \eqref{20'} of the Maxwell equations, we rewrite system \eqref{100} as
\begin{equation}
\varepsilon\mathcal{B}_0\partial_t{V}+\widehat{\mathcal{B}}_1\partial_1{V}+\mathcal{B}_2\partial_2{V}+
\mathcal{B}_3\partial_3{V}=0 \qquad \mbox{in}\ \Omega_T,\label{105}
\end{equation}
where $\widehat{\mathcal{B}}_1=\varepsilon\varkappa \mathcal{B}_0-\mathcal{B}_1$, and our choice of $\nu_i$ is the following:
\begin{equation}
\nu =\varepsilon\hat{v}= \varepsilon (\hat{v}_1 ,\hat{v}_2,\hat{v}_3)
\label{106}
\end{equation}
($\hat{v}_1=\varkappa$, see \eqref{98}). Since $\varepsilon$ is very small, the hyperbolicity condition \eqref{21'} holds. By standard arguments of the energy method applied to the symmetric hyperbolic systems \eqref{99} and \eqref{105} (we multiply \eqref{105} by $\varepsilon^{-1}$), we obtain
\begin{equation}
I(t)+ 2\int_{\partial\Omega_t}\mathcal{Q}\,{\rm d}x'{\rm d}s\leq C
\left( \| F\|^2_{L_2(\Omega_T)} +\int_0^tI(s)\,{\rm d}s\right),
\label{107}
\end{equation}
where
\[
I(t)=\int_{\mathbb{R}^3_+}\left(|U|^2+(\mathcal{B}_0V,V)\right){\rm d}x,
\quad
\mathcal{Q}=-\frac{1}{2}(\widehat{A}_1U,U)|_{x_1=0}-\frac{1}{2\varepsilon}(\widehat{\mathcal{B}}_1V,V)|_{x_1=0}.
\]

We write down the quadratic form $\mathcal{Q}$:
\begin{align*}
\mathcal{Q}= &\,\bigl\{-\varkappa(\mathcal{H}_2^2+\mathcal{H}_3^2+ E_2^2+ E_3^2)+\mathcal{H}_1(\hat{v}_2\mathcal{H}_2+\hat{v}_3\mathcal{H}_3)
+E_1(\hat{v}_2E_2+\hat{v}_3E_3)
\\
&\;\;
+\varepsilon\varkappa E_1(\hat{v}_3\mathcal{H}_2-\hat{v}_2\mathcal{H}_3)
+\varepsilon\varkappa \mathcal{H}_1(\hat{v}_2E_3-\hat{v}_3E_2)\\
&\;\;
+(\varepsilon^{-1}+\varepsilon\varkappa^2)(\mathcal{H}_3E_2-\mathcal{H}_2E_3) -qv_1\bigr\}\bigr|_{x_1=0}\,.
\end{align*}
Using the boundary conditions \eqref{101} and the second condition in \eqref{104}, after long calculations (which are similar to those in \cite{T12}) we get
\[
\mathcal{Q}=\hat{\mu}\left.\left\{ E_1\partial_t\varphi +(\varepsilon^{-1}\mathcal{H}_2+\varkappa E_3)\partial_3\varphi -(\varepsilon^{-1}\mathcal{H}_3-\varkappa E_2)\partial_2\varphi \right\}\right|_{x_1=0},
\]
where
\begin{equation}
\hat{\mu}=\widehat{E}_1+\varepsilon\hat{v}_2\widehat{\mathcal{H}}_3-\varepsilon\hat{v}_3\widehat{\mathcal{H}}_2=\widehat{E}_1+\mathcal{O}(\varepsilon ).
\label{mu}
\end{equation}
We rewrite $\mathcal{Q}$ as follows:
\[
\mathcal{Q}=\partial_t\left(\hat{\mu}\varphi E_1|_{x_1=0} \right)+\partial_2\left(\hat{\mu}\varphi(\varkappa E_2-\varepsilon^{-1}\mathcal{H}_3)|_{x_1=0} \right)  + \partial_3\left(\hat{\mu}\varphi (\varepsilon^{-1}\mathcal{H}_2+\varkappa E_3)|_{x_1=0}\right) +\mathcal{Q}_0,
\]
and in view the third equation in \eqref{103}, the rest $\mathcal{Q}_0=0$  because it is the left-hand side of the fourth equation in \eqref{100} considered on the boundary and multiplied by $-\varepsilon^{-1}\hat{\mu}\varphi$:
\[
\mathcal{Q}_0=-\hat{\mu}\varphi \left.\left( \partial_tE_1+\varkappa\partial_1E_1+\varepsilon^{-1}\partial_3\mathcal{H}_2-
\varepsilon^{-1}\partial_2\mathcal{H}_3 \right)\right|_{x_1=0}=0.
\]

It follows from \eqref{107} that
\begin{equation}
I(t)+ 2\int_{\mathbb{R}^2}\hat{\mu}\varphi E_1|_{x_1=0}{\rm d}x'\leq C
\left( \| F\|^2_{L_2(\Omega_T)} +\int_0^tI(s)\,{\rm d}s\right),
\label{109}
\end{equation}
and we see that we are not able to ``close'' the estimate in $L_2$. But, if we differentiate systems \eqref{99} and \eqref{105} with respect to $x^2$, $x^3$ and $t$, we obtain the following counterparts of \eqref{109} for the first-order tangential derivatives of $U$ and $V$:
\begin{equation}
I_{\alpha}(t)+ 2\int_{\mathbb{R}^2}\hat{\mu}\,\partial_{\alpha}\varphi \, \partial_{\alpha}E_1|_{x_1=0}{\rm d}x'\leq C
\left( [F]^2_{1,*T} +\int_0^tI_{\alpha}(s)\,{\rm d}s\right),
\label{110}
\end{equation}
where $\alpha =0,2,3$,
\[
I_{\alpha}(t)=\int_{\mathbb{R}^3_+}\left(|\partial_{\alpha}U|^2+(\mathcal{B}_0\partial_{\alpha}V,\partial_{\alpha}V)\right){\rm d}x\quad \mbox{and}\quad \partial_0:=\partial_t.
\]
The terms $\partial_{\alpha}\varphi \, \partial_{\alpha}E_1|_{x_1=0}$ appearing in the boundary integral are, in some sense, lower-order terms and below we explain how to treat them by using our important assumption \eqref{52a} and passing to the volume integral.

Thanks to  \eqref{52a}, from \eqref{104} and the first boundary condition in \eqref{101} we get
\begin{equation}
\left\{
\begin{array}{l}
{\displaystyle
\partial_t\varphi =a_1^0{H}_1|_{x_1=0}+a_2^0{\mathcal{H}}_1|_{x_1=0} +v_1|_{x_1=0},}\\[12pt]
{\displaystyle
\partial_2\varphi =\frac{(\widehat{\mathcal{H}}_3{H}_1-\widehat{H}_3{\mathcal{H}}_1)|_{x_1=0}
}{\widehat{H}_2\widehat{\mathcal{H}}_3-\widehat{H}_3\widehat{\mathcal{H}}_2}=
a_1^1{H}_1|_{x_1=0}+a_2^1{\mathcal{H}}_1|_{x_1=0},}\\[12pt]
{\displaystyle
\partial_3\varphi=-\frac{(\widehat{\mathcal{H}}_2{H}_1-\widehat{H}_2{\mathcal{H}}_1)|_{x_1=0}}{\widehat{H}_2\widehat{\mathcal{H}}_3-\widehat{H}_3\widehat{\mathcal{H}}_2}=
a_1^2{H}_1|_{x_1=0}+a_2^2{\mathcal{H}}_1|_{x_1=0},}
\end{array}\right.
\label{111}
\end{equation}
where the constants  $a_j^{\beta}$ can be easily written down. Then \eqref{111} implies
\begin{multline}
(\partial_t\varphi \, \partial_tE_1+\partial_2\varphi \, \partial_2E_1+\partial_3\varphi \, \partial_3E_1)|_{x_1=0}\\
=v_1\partial_tE_1|_{x_1=0}+\sum_{\beta =0,2,3}\left. \left(a_1^{\beta}{H}_1\partial_{\beta}E_1+a_2^{\beta}{\mathcal{H}}_1\partial_{\beta}E_1\right)\right|_{x_1=0}\,.
\label{115}
\end{multline}

To treat the integrals of the lower-order terms like $H_1\partial_kE_1|_{x_1=0}$, $k=2,3$, contained in the right-hand side of \eqref{115} we use the same standard arguments as in \cite{T09,T09.cpam,T12}, i.e., we pass to the volume integral and integrate by parts:
\[
\int_{\mathbb{R}^2}H_1\partial_kE_1|_{x_1=0}\,{\rm d}x'=-\int_{\mathbb{R}^3_+}\partial_1(H_1\partial_kE_1)\,{\rm d}x  =
\int_{\mathbb{R}^3_+}\left\{ \partial_kH_1\partial_1E_1-\partial_1H_1\partial_kE_1 \right\}{\rm d}x.
\]
Concerning the terms like $H_1\partial_tE_1|_{x_1=0}$, we again pass to the volume integral, but before integration by parts we apply the third equation in \eqref{103}:
\begin{multline*}
\int_{\mathbb{R}^2}H_1\partial_tE_1|_{x_1=0}\,{\rm d}x'=-\int_{\mathbb{R}^3_+}\partial_1(H_1\partial_tE_1)\,{\rm d}x  =
-\int_{\mathbb{R}^3_+}\partial_1H_1\partial_tE_1 {\rm d}x \\
-\int_{\mathbb{R}^3_+}H_1 \partial_t\partial_1E_1 {\rm d}x=\int_{\mathbb{R}^3_+}\left\{ \partial_2H_1\partial_tE_2+\partial_3H_1\partial_tE_3-\partial_1H_1\partial_tE_1 \right\}{\rm d}x.
\end{multline*}
After that the normal ($x_1$-) derivatives of $v_1$, $H_1$ and $V$ can be expressed through tangential derivatives by using the facts that the boundary $x_1=0$ is noncharacteristic for the ``vacuum'' system \eqref{100} and $v_1$ and $H_1$ are noncharacteristic ``plasma'' unknowns.

Namely, taking into account \eqref{bmq}, \eqref{35'} and the first equation in \eqref{103}, we have
\begin{align}
\left(\begin{array}{c}
\partial_1v_1 \\[3pt]
\partial_1q
\end{array}
\right) &= -\left\{J^{\sf T}\left(\partial_t{U}+\widehat{A}_2\partial_2{U}+\widehat{A}_3\partial_3{U}+F\right)\right\}_{(12)},\label{118}
\\
\partial_1H_1 &=-\partial_2H_2-\partial_3H_3,\label{119}
\end{align}
where $\{\cdots\}_{(12)}$ denotes the first two components of the vector inside braces.\footnote{In fact, below we use only the first row in \eqref{118}.} As in \cite{T12},
we could resolve the ``vacuum'' system for the normal derivative $\partial_1V$ under assumption \eqref{expan.lin} guaranteeing $\det\widehat{B}_1\neq 0$.
However, this would lead to the appearance of $\varepsilon^{-1}$ as a coefficient by tangential derivatives of $V$ and give a more restrictive sufficient stability condition (with a constant $\widehat{E}_1^*$ of order less or equal to $\varepsilon$ in Theorem \ref{t1}). To avoid this we replace the first and fourth equations in system \eqref{100} with the second and third equations in \eqref{103} and then from the resulting system express $\partial_1V$ through tangential derivatives of $V$:
\begin{equation}
\partial_1{V}=\frac{1}{1-\varepsilon^2\varkappa^2}\begin{pmatrix}
(1-\varepsilon^2\varkappa^2)(\partial_2\mathcal{H}_2 +\partial_3\mathcal{H}_3)\\ \varepsilon\varkappa (\partial_t\mathcal{H}_2+\partial_3E_1)-\partial_2\mathcal{H}_1-\partial_tE_3\\
\varepsilon\varkappa (\partial_t\mathcal{H}_3-\partial_2E_1)-\partial_3\mathcal{H}_1+\partial_tE_2\\
(1-\varepsilon^2\varkappa^2)(\partial_2E_2 +\partial_3E_3)\\
\varepsilon\varkappa (\partial_tE_2-\partial_3\mathcal{H}_1)-\partial_2E_1+\partial_t\mathcal{H}_3 \\
\varepsilon\varkappa (\partial_tE_3+\partial_2\mathcal{H}_1)-\partial_3E_1-\partial_t\mathcal{H}_2
\end{pmatrix}.
\label{117}
\end{equation}
Note that \eqref{117} is valid also for the case $\varkappa =0$ (cf. Remark \ref{r2}).

Then, \eqref{115}--\eqref{117} and the above calculations yield
\begin{equation}
2\sum_{\alpha =0,2,3}\int_{\mathbb{R}^2}\hat{\mu}\partial_{\alpha}\varphi \, \partial_{\alpha}E_1|_{x_1=0}{\rm d}x'  =\int_{\mathbb{R}^3_+}\hat{\mu}({\rm Q}_0Z,Z){\rm d}x -2\int_{\mathbb{R}^3_+}\hat{\mu}F_1\partial_tE_1{\rm d}x,
\label{120}
\end{equation}
where ${\rm Q}_0$ is a quadratic matrix of order 42 which elements can be explicitly written down if necessary and  $Z= (\partial_tU,\partial_tV,\partial_2U,\partial_2V,\partial_3U,\partial_3V)$.

Summing up inequalities \eqref{110} and taking into account \eqref{120}, we obtain
\begin{equation}
\int_{\mathbb{R}^3_+}\left( (\mathfrak{A}_0+ \hat{\mu}{\rm Q}_0)Z,Z\right){\rm d}x
-2\int_{\mathbb{R}^3_+}\hat{\mu}F_1\partial_tE_1{\rm d}x
\leq C
\left( [F]^2_{1,*T} +\sum_{\alpha =0,2,3}\int_0^tI_{\alpha}(s)\,{\rm d}s\right),
\label{121}
\end{equation}
where $\mathfrak{A}_0=\diag (I,{\mathcal{B}}_0,\ldots ,I,{\mathcal{B}}_0)$ is the block-diagonal positive definite matrix of order 42. Using the Young inequality and the elementary inequality
\[
\| F(t)\|^2_{L_2(\mathbb{R}^3_+)}\leq C [F]^2_{1,*,t}
\]
following from the trivial relation
\[
\frac{\rm d}{{\rm d}t}\,\|F(t)\|^2_{L_2(\mathbb{R}^3_+)}=2\int_{\mathbb{R}^3_+}(F,\partial_tF)\,{\rm d}x,
\]
we estimate:
\begin{equation}
2\int_{\mathbb{R}^3_+}\hat{\mu}F_1\partial_tE_1{\rm d}x\leq C\left\{ \delta\, \|Z(t)\|^2_{L_2(\mathbb{R}^3_+)} + \frac{1}{\delta}\,\| F\|^2_{1,*,T} \right\},
\label{124}
\end{equation}
where $\delta$ is a small positive constant, and \eqref{121} and \eqref{124} give
\begin{equation}
\int_{\mathbb{R}^3_+}\left( (\mathfrak{A}_0+ \hat{\mu}{\rm Q}_0)Z,Z)\right){\rm d}x
\leq C
\left( \delta\, \|Z(t)\|^2_{L_2(\mathbb{R}^3_+)} +[F]^2_{1,*T} +\|Z\|^2_{L_2(\Omega_t)}\right).
\label{121'}
\end{equation}

If $\mathfrak{A}_0+ \hat{\mu}{\rm Q}_0>0$ or, in view of \eqref{mu}, if
\begin{equation}
\mathfrak{A}_0+ \widehat{E}_1{\rm Q}_0>0,
\label{122}
\end{equation}
then choosing $\delta$ to be small enough, from inequality  \eqref{121'} we derive
\begin{equation}
\|Z(t)\|^2_{L_2(\mathbb{R}^3_+)}
\leq C
\left( [F]^2_{1,*T} +\|Z\|^2_{L_2(\Omega_t)}\right).
\label{125}
\end{equation}
We can already apply to \eqref{125} Gronwall's lemma and get an a priori estimate for tangential derivatives. After that it is easy to insert the $L_2$ norm of the solution in this estimate and get finally an estimate in a conormal Sobolev space. However, for adapting the arguments of this section to the case of variable coefficients it is better to follow the plan of \cite{T12} and by using \eqref{125} obtain the inequality
\begin{multline}
\nt U(t)\nt ^2_{1,*}+\nt V(t)\nt ^2_{H^1(\mathbb{R}^3_+)}+
\nt\varphi (t)\nt^2_{H^{3/2}(\mathbb{R}^2)} \\ \leq C\Bigl\{ [F]^2_{1,*,T}
+  [{U}]^2_{1,*,t} +\|V\|^2_{H^1(\Omega_t)}+\|\varphi\|^2_{H^{3/2}(\partial\Omega_t)}\Bigr\}.
\label{134}
\end{multline}
We drop the arguments towards the proof of \eqref{134} because they are relatively standard and absolutely similar to those in \cite{T12}. In fact, here we  improve the corresponding inequality from \cite{T12} by replacing the $H^1$ norm of $\varphi$ with its $H^{3/2}$ norm. Taking into account \eqref{111}, this can be easily done because we control $H^{1/2}$ norms of the traces of the noncharacteristic unknowns $v_1$, $H_1$ and $\mathcal{H}_1$. Applying Gronwall's lemma, from \eqref{134} we derive the desired a priori estimate \eqref{54}, provided that the {\it sufficient stability condition} \eqref{122} holds.

Since the matrix $\mathfrak{A}_0$ is positive definite, condition \eqref{122} is satisfied for $|\widehat{E}_1|\ll 1$. Clearly, in the general case it is technically impossible to analyze the positive definiteness of the matrix $\mathfrak{A}_0+ \widehat{E}_1{\rm Q}_0$ of order 42 analytically. Here we do this for the particular case \eqref{pcase}. In this case the coefficients in \eqref{111} read
\[
a_1^3=\ds\frac{1}{\wh{H}_3},\quad a_2^2=\ds\frac{1}{\wh{\vH}_2},\quad
a_1^0=-\ds\frac{\hat{v}_3}{\wh{H}_3},\quad a_2^3=0,\quad a_1^2=0,\quad a_2^0=0,
\]
and, for example, \eqref{118} has the simple form
\[
\pr{1}v_1= - \pr{t}p -\pr{2}v_2 - \pr{3} v_3 - \hat{v}_3 \pr{3}p,\qquad \pr{1}q= - \pr{t}v_1 -\hat{v}_3\pr{3}v_1 + \wh{H}_3\pr{3} H_1
\]
and we have \eqref{117} with $\varkappa =0$.

After long calculations we get that the positive definiteness of the matrix $\mathfrak{A}_0+ \widehat{E}_1{\rm Q}_0$ is equivalent to the fact that all of the roots of the following polynomial are positive:
\begin{multline}
P(x)=\left\{(1-x)^2 -
2\ds\frac{\wh{E}_1^2}{\wh{\vH}{}_2^2}\right\}\left\{(1-x)^2-\wh{E}_1^2\left(1+
\ds\frac{\hat{v}{}_3^2}{\wh{H}{}_3^2}\right)\right\}\\
\times\left\{ \left((1-x)^2-\wh{E}_1^2\right)\left((1-x)^2-
2\wh{E}_1^2\ds\frac{\wh{\vH}{}_2^2+ \wh{H}{}_3^2}{\wh{\vH}{}_2^2
\wh{H}{}_3^2} \right)-\ds\frac{\hat{v}{}_3^2 \wh{E}_1^2}{\wh{H}{}_3^2}
\left((1-x)^2-2\ds\frac{\wh{E}_1^2}{\wh{\vH}{}_2^2}\right)
\right\}\\
\times \left\{ (1-x)^2\left((1-x)^2-2\wh{E}_1^2
\ds\frac{1+\hat{v}{}_3^2}{\wh{H}{}_3^2}\right)-\wh{E}_1^2 (3+
\hat{v}{}_3^2)\left((1-x)^2-2\ds\frac{\wh{E}_1^2}{\wh{H}{}_3^2}\right)
\right\}.
\label{px}
\end{multline}
We see that the polynomial $P(x)$ can be considered as a polynomial of $y=(1-x)^2$. Since the matrix $\mathfrak{A}_0+ \widehat{E}_1{\rm Q}_0$ is symmetric, all of its eigenvalues are real and, hence, all of the roots of the polynomial $Q(y):=P(x)$ should be positive. One can show that this is indeed so and, therefore, the requirement $x>0$ is equivalent to $y<1$. The polynomial $Q(y)$ is the multiplication of four polynomials (four expressions in braces in \eqref{px}). The requirement that each of these four polynomials has only roots less than unit is equivalent to the four inequalities in \eqref{posdef}. In the static case $\hat{v}_3=0$ this gives condition \eqref{v3=0}, and domains 3 in Fig. 1 describe this condition for different fixed values of $\widehat{H}_3$ in the plane of parameters $\widehat{E}_1$ and $\wh{\mathcal{H}}_2$.

\begin{remark}
{\rm
Under assumption \eqref{52a} the interface symbol is {\it elliptic}, i.e., we are able to resolve the boundary conditions (\eqref{101} together with the boundary constraints \eqref{104}) for the gradient $\nabla_{t,x}\varphi=(\partial_t\varphi,\partial_2\varphi,\partial_3\varphi)$, cf. \eqref{111}. This assumption is really crucial for deriving the a priori estimate. On the other hand, the interface symbol can be elliptic even if \eqref{52a} is violated. Indeed, substituting the first condition in \eqref{101} into the third and fourth ones and using the second condition in \eqref{104}, we obtain
\begin{equation}
\begin{array}{c}
\hat{\mu}\partial_2\varphi =-E_2|_{x_1=0}+\varepsilon \bigl(\wh{\vH}_3v_1-\hat{v}_3\vH_1+\varkappa\vH_3\bigr)|_{x_1=0},\\[3pt]
\hat{\mu}\partial_3\varphi =-E_3|_{x_1=0}+\varepsilon \bigl(-\wh{\vH}_2v_1+\hat{v}_2\vH_1-\varkappa\vH_2\bigr)|_{x_1=0},
\end{array}
\label{r3.1}
\end{equation}
and the substitution of \eqref{r3.1} into the first condition in \eqref{101} gives
\begin{equation}
\hat{\mu}\partial_t\varphi =\wh{E}_1v_1|_{x_1=0}+\hat{v}_2E_2|_{x_1=0}+\hat{v}_3E_3|_{x_1=0} +\varepsilon \varkappa \bigl(\hat{v}_3{\vH}_2-\hat{v}_2\vH_3\bigr)|_{x_1=0}.
\label{r3.2}
\end{equation}
That is, if $\hat{\mu}\neq 0$, then the interface symbol is elliptic.\footnote{If $\hat{\mu}= 0$, then we can get an a priori estimate for the case of constant coefficients even in $L_2$ (cf. \eqref{109}), but for the case of variable coefficients the assumption about ellipticity of the interface symbol is really necessary for obtaining an estimate (see \cite{T12}).} Note that without taking into account the second condition in \eqref{104} we can resolve the boundary conditions \eqref{101} for $\nabla_{t,x}\varphi$ under the more restrictive condition $\widehat{E}_1\hat{\mu}\neq 0$.

Substituting \eqref{r3.1} and \eqref{r3.2} into \eqref{110} for corresponding $\alpha=0,2,3$ and repeating the arguments towards getting equality \eqref{120} (passing to volume integrals, integrating by parts, etc.), we obtain a counterpart of \eqref{120}, where the matrix $\hat{\mu}{\rm Q}_0$ is replaced by some symmetric matrix $\widetilde{\rm Q}_0$ which can be explicitly calculated. Then the quadratic form $((\mathfrak{A}_0+ \hat{\mu}{\rm Q}_0)Z,Z)$ in \eqref{121} is replaced by the form $((\mathfrak{A}_0+ \widetilde{\rm Q}_0)Z,Z)$. Omitting simple calculations connected with the reduction of the boundary integrals $\int_{\mathbb{R}^2}E_k\partial_kE_1|_{x_1=0}{\rm d}x'$ for $k=2,3$ to volume integrals of tangential derivatives of $E$, we get that the quadratic form for the vector $(\partial_2E_1,\partial_3E_1,\partial_2E_2,\partial_3E_3)$ contained in $((\mathfrak{A}_0+ \widetilde{\rm Q}_0)Z,Z)$ reads
\[
-\left\{ (\partial_2E_1)^2+(\partial_3E_1)^2+(\partial_2E_2)^2+(\partial_3E_3)^2+4\partial_2E_2\partial_3E_3\right\}.
\]
Hence, the whole quadratic form $((\mathfrak{A}_0+ \widetilde{\rm Q}_0)Z,Z)$ cannot be positive definite. We have similar situation if for finding $\nabla_{t,x}\varphi$ we use \eqref{101} together with the first condition in \eqref{104} (but not the second one as above).

Thus, we are not able to derive an a priori estimate by using our energy method for the case when assumption \eqref{52a} is violated (or if instead of \eqref{52a} we just apply other ellipticity conditions for the interface symbol). Moreover, if $\widehat{H}_2\widehat{\mathcal{H}}_3-\widehat{H}_3\widehat{\mathcal{H}}_2=0$, then, at least for a particular case, in the next section we prove that the linearized problem is ill-posed for any $\widehat{E}_1\neq 0$.
}
\label{r3}
\end{remark}

\section{Normal modes analysis for particular cases}
\label{sec:5}

We seek exponential solutions to problem \eqref{99}--\eqref{101} with $F=0$ and special initial data:
\begin{align}
& U = \bar{U}\exp \left\{ \tau t +\xi_{\rm p}x_1+ i(\gamma ', x') \right\},\label{135}\\
& V = \bar{V}\exp \left\{ \tau t +\xi_{\rm v}x_1+ i(\gamma ', x') \right\},\label{136}\\
& \varphi = \bar{\varphi}\exp \left\{ \tau t + i(\gamma ', x') \right\},\label{137}
\end{align}
where $\bar{U}$ and $\bar{V}$ are complex-valued constant vectors, $\tau$, $\xi_{\rm p}$, $\xi_{\rm v}$ and $\bar{\varphi}$ are complex constants, and $\gamma '=(\gamma_2,\gamma_3)$ with real constants $\gamma_{2,3}$. The existence of exponentially growing solutions in
form \eqref{135}--\eqref{137}, with
\begin{equation}
\Re\,\tau>0,\quad \Re\,\xi_{\rm p}<0,\quad
\Re\,\xi_{\rm v}<0,\label{138}
\end{equation}
implies the ill-posedness of problem \eqref{99}--\eqref{101} because the consequence of solutions
\[
(U (nt,n{x} ), V( nt, nx), \varphi (nt,n{x}' ))\exp
(-\sqrt{n}\,),\quad n=1,2,3,\ldots\,,
\]
with $U$, $V$ and $\varphi$ defined in \eqref{135}--\eqref{137}, is the Hadamard-type ill-posedness example. Since the last equation in \eqref{99} for the entropy perturbation $S$ plays no role in the construction of an ill-posedness example, we will suppose that $U=(p,u,H)$.

Theoretically, we can construct a 1D ill-posedness example (with $\gamma' =0$) on a codimension-1 set in the parameter domain. But, this is not the case for problem \eqref{99}--\eqref{101}. Indeed, in 1D we have $E_1=0$ and the energy inequality \eqref{109} implies well-posedness. Thus, we may assume that $\gamma' \neq 0$ and even, without loss of generality, $|\gamma'|=1$.

\subsection{Analytical study}

We analytically construct a 2D ill-posedness example with $\gamma_3=0$ (i.e., $\gamma_2 =1$) for some particular cases of the unperturbed flow. Since the assumption $\gamma_3=0$ is a restriction, this will give us {\it sufficient} instability conditions.

We will consider the unperturbed flow with
\begin{equation}
\hat{v}_1=\hat{v}_2=\widehat{H}_2=0
\label{unfl}
\end{equation}
(recall that $\widehat{H}_1=0$, see \eqref{98}). Since $\hat{v}_1=\varkappa$, we have $\varkappa =0$. However, by continuity the below ill-posedness examples can be extended to the case $|\varkappa | \ll 1$. Substituting \eqref{135} and \eqref{136} into \eqref{99} (with $F=0$) and \eqref{100} respectively, we get the dispersion relations
\[
\det (\tau I  + \xip \wh{A}_1 + i \wh{A}_2)=0\qquad\mbox{and}\qquad \det (\ve\tau I + \xiv \wh{B}_1 + i B_2)=0,
\]
from which, omitting technical calculations, we find the unique roots
\begin{equation}
\xip = - \dsq{1+K\tau^2},\qquad \xiv=-\dsq{1+\ve^2\tau^2},
\label{139}
\end{equation}
with property \eqref{138}, where \(K=1/(1+\wh{H}{}_3^2)\).

Consider first the particular case $\widehat{\mathcal{H}}_2=0$ of the unperturbed flow \eqref{unfl}. For this case our basic assumption \eqref{52a} is violated. Substituting \eqref{135}--\eqref{137} into the boundary conditions \eqref{101}, we obtain
\[
    \tau \bar{\vfi} = \bar{v}_1, \quad\bar{q} = \wh{\vH}_3
    \bar{\vH}_3 - \wh{E}_1 \bar{E}_1, \quad \bar{E}_2 = \ve \tau
    \bar{\vfi} \wh{\vH}_3 - i \bar{\vfi} \wh{E}_1, \notag
\]
and excluding the constant \(\bar{\vfi}\) we come to the relations
\begin{equation}
\bar{p}+ \bar{H}_3 \wh{H}_3 +\bar{E}_1 \wh{E}_1 - \bar{\vH}_3
\wh{\vH}_3=0,\quad (i \wh{E}_1 - \tau \ve \wh{\vH}_3)\bar{v}_1 +
\tau \bar{E}_2 =0.\label{140}
\end{equation}
From \eqref{135} and \eqref{136} we have
\begin{equation}
\begin{array}{c}
\tau \bar{p} + \xip \bar{v}_1 + \ds\frac{1}{\tau} (\bar{\vH}_3
\wh{\vH}_3 - \bar{E}_1 \wh{E}_1) = 0,\\[6pt]\tau \bar{v}_1 + \xip
(\bar{p}+\bar{H}_3 \wh{H}_3)=0 , \quad \tau \ve \bar{E}_1 = i
\bar{\vH}_3,\quad \tau \ve \bar{E}_2 = \xiv
\bar{\vH}_3.\end{array}\label{141}
\end{equation}
Then, \eqref{140} and \eqref{141} imply
\begin{equation}
\begin{pmatrix}
\tau & \xip & -\ds\frac{1}{\tau}(\wh{E}_1 + i \tau \ve
\wh{\vH}_3)\\[3pt]
0& \tau & - \xip (\wh{E}_1 + i \tau \ve \wh{\vH}_3)\\[3pt]
0& i\wh{E}_1 - \tau \ve \wh{\vH}_3 & -\tau \xiv
\end{pmatrix}
\begin{pmatrix}
\bar{p} \\ \bar{v}_1 \\ \bar{E}_1
\end{pmatrix}
=0.\label{142}
\end{equation}

Since we are interesting in solutions with $(\bar{U},\bar{V})\neq 0$, the determinant of the matrix in \eqref{142} should be equal to zero. Taking into account \eqref{139}, this gives us the following final equation for $\tau$:
\begin{equation}
\tau^2 \dsq{1+ \ve^2 \tau^2} = (\wh{E}_1 + i \tau \ve
\wh{\vH}_3)^2 \dsq{1+ K
\tau^2}.\label{143}
\end{equation}
If $\widehat{E}_1=0$, then \eqref{143} has only roots with $\Re\,\tau =0$. That is, for the 2D problem (recall that $\gamma_3=0$) the Kreiss-Lopatinski is satisfied, but only in a weak sense because the Lopatinski determinant has imaginary roots. At the same time, since we assume that $\hat{v}_2=0$ (see \eqref{unfl}), for the case when $\widehat{\mathcal{H}}_2=\widehat{E}_1=0$ the constant $\hat{\mu}$ in \eqref{mu} is zero. Then \eqref{109} implies an $L_2$ a priori estimate for problem \eqref{99}--\eqref{102}, i.e., the Kreiss-Lopatinski is satisfied  (in 3D) at least in a weak sense.\footnote{Our hypothesis is that in the general case the Kreiss-Lopatinski can be satisfied {\it only} in a weak sense, i.e., the uniform Kreiss-Lopatinski is is always violated for problem \eqref{99}--\eqref{102}, but this is hard to be proved analytically.}

If $\widehat{E}_1\neq 0$, then, setting formally $\varepsilon =0$,  graphically it is easy to see that equation \eqref{143} has a root $\tau^2>0$, i.e., we have an ``unstable'' root $\tau >0$. Since $\varepsilon$ is a very small (fixed) constant, it is clear that there is an ``unstable'' root $\tau$ with $\Re\,\tau >0$  for $\varepsilon \neq 0$. Thus, we see that if assumption \eqref{52a} is violated, then the linearized problem can be ill-posed (see Remark \ref{r3}).

We now consider a particular case when assumption \eqref{52a} holds. Let $\widehat{\mathcal{H}}_3=0$ and $\widehat{H}_3\widehat{\mathcal{H}}_2\neq 0$. In view of \eqref{unfl}, this is nothing else than the particular case \eqref{pcase}. Substituting \eqref{135}--\eqref{137} into the boundary conditions \eqref{101}, after the exclusion of
constant $\bar{\varphi}$ we obtain
\begin{equation}
\bar{p}+ \bar{H}_3 \wh{H}_3 +\bar{E}_1 \wh{E}_1 - \bar{\vH}_2
\wh{\vH}_2=0,\quad \tau \bar{E}_2 + i \bar{v}_1 \wh{E}_1 = 0 ,\quad
\bar{E}_3+ \ve \bar{v}_1 \wh{\vH}_2 =0.\label{145}
\end{equation}
It follows from \eqref{99} and \eqref{100} that
\begin{equation}
\begin{array}{c}
\tau \bar{p} + \xip \bar{v}_1 + \ds\frac{1}{\tau} (\bar{\vH}_2
\wh{\vH}_2 - \bar{E}_1 \wh{E}_1) = 0,\quad \tau \bar{v}_1 + \xip
(\bar{p}+\bar{H}_3 \wh{H}_3)=0 ,\\[6pt]
\tau \ve \bar{\vH}_2 = - \xiv \bar{E}_3,\quad\tau \ve \bar{E}_2 =
\xiv \bar{\vH}_3,\quad\tau \ve \bar{E}_1 = i
\bar{\vH}_3.
\end{array}
\label{146}
\end{equation}
From \eqref{145} and \eqref{146} we get
\begin{equation}
\begin{pmatrix}
 \tau & (\xip + \ds\frac{\xiv \wh{\vH}{}_2^2}{\tau^2}) &
-\ds\frac{1}{\tau} \wh{E}_1\\[9pt]
 0& (\tau + \xip \ds\frac{\xiv \wh{\vH}{}_2^2}{\tau}) &
 -\xip \wh{E}_1\\
 0& \wh{E}_1 & - \tau \xiv
\end{pmatrix}
\begin{pmatrix}
 \bar{p}\\\bar{v}_1\\\bar{E}_1
\end{pmatrix}
=0.\label{146'}
\end{equation}

System \eqref{146'} has anonzero solution $(\bar{p},\,\bar{v}_1,\,\bar{E}_1)$ if
\begin{equation}
\tau^2 \dsq{1+ \ve^2 \tau^2} = (\wh{E}{}_1^2 - \wh{\vH}{}_2^2 (1+
\ve^2 \tau^2)) \dsq{1+ K
\tau^2}\label{147}
\end{equation}
or
\[
 G(y) = y \dsq{1+\ve^2 y} - (\wh{E}{}_1^2 - \wh{\vH}{}_2^2 (1+
\ve^2 y))\dsq{1+K y} =0,
\]
with $y=\tau^2$. If $\widehat{E}_1^2>\widehat{\mathcal{H}}_2^2$, then $G(0)<0$ and $G(y^*)>0$, where
\[
y^*=\frac{\wh{E}{}_1^2 -
\wh{\vH}{}_2^2}{\ve^2 \wh{\vH}{}_2^2} >0
\]
(recall that we consider the case $\widehat{\mathcal{H}}_2\neq 0$). Hence, if \eqref{anins} holds, then equation \eqref{147} has an ``unstable'' root $\tau >0$, i.e., problem \eqref{99}--\eqref{102} is ill-posed. If \eqref{anins} is violated, taking into account that $\varepsilon$ is small enough, it is easy to see that \eqref{147} has only roots with $\Re\,\tau =0$. That is, again the Kreiss-Lopatinski can be satisfied only in a weak sense (at least, in 2D).

\subsection{Numerical investigation}

It should be noted that condition \eqref{anins} is only sufficient for instability for the particular case \eqref{pcase} because  normal modes analysis was performed under the restriction $\gamma_3=0$ on the wave vector. For finding a necessary and sufficient instability condition we have to repeat our above analysis for an arbitrary wave vector $\gamma'$ (with $|\gamma'|=1$). Since it is technically impossible to do this analytically, we do numerical calculations by using MATLAB$^\circledR$ software.

Just for technical simplicity we consider the static case $\hat{v}_3=0$. Repeating the arguments above, we obtain the following counterparts of \eqref{139} and the Lopatinski determinant \eqref{147} for the general case of the wave vector \(\gamma = (\gamma_2 ,\,\, \gamma_3) = ( \cos\psi ,
\sin\psi ) \):
\[
\xip = - \dsq{1 + \ds\frac{\tau^4}{(1+ \wh{H}{}_3^2)\tau^2 +
\sin^2\psi\, \wh{H}{}_3^2}},\qquad \xiv = -\dsq{1+\ve^2\tau^2},
\]
\begin{equation}
(\tau^2 + \wh{H}{}_3^2\sin^2\psi ) \xiv =  \left(\wh{E}{}_1^2 -
\wh{\vH}{}_2^2 (\ve^2\tau^2  +  \cos^2\psi) - 2 i \ve \tau
 \wh{E}_1 \wh{\vH}_2\sin\psi\right)\xip .\label{ld}
\end{equation}

We fix $\ve =10^{-6}$ and four different values of $\wh{H}{}_3$: $\wh{H}{}_3=1$, $\wh{H}{}_3=2/3$, $\wh{H}{}_3=0.5$ and $\wh{H}{}_3=0.25$. Then, we choose the partition of the interval $(0, 2\pi )$ with the step $10^{-2}$ for the angle $\psi$ and for all its points numerically solve equation \eqref{ld} for $\tau$. The results of these calculations in the plane of parameters $\widehat{E}_1$ and $\wh{\mathcal{H}}_2$ are presented in Fig. 1, where the union of domain 1 and 4 is the whole instability domain for case \eqref{pcase} with $\hat{v}_3=0$. Recall that domain 1 is the domain described by the sufficient instability condition \eqref{anins} found analytically.

\section{Well-posedness of the linearized problem}
\label{sec:6}

For the case of variable coefficients the counterpart of the secondary symmetrization \eqref{105} for the linearized ``vacuum'' system reads
\begin{equation}
\varepsilon\mathcal{B}_0\partial_t{V}+\widehat{\mathcal{B}}_1\partial_1{V}+\mathcal{B}_2\partial_2{V}+
\mathcal{B}_3\partial_3{V}=0 \qquad \mbox{in}\ \Omega_T,\label{105v}
\end{equation}
where
\[
\widehat{\mathcal{B}}_1= \frac{1}{\partial_1\widehat{\Phi}^-}\left(\mathcal{B}_1-\varepsilon I\,\partial_t\widehat{\Psi}^- -\mathcal{B}_2\partial_2\widehat{\Psi}^-
-\mathcal{B}_3\partial_3\widehat{\Psi}^- \right).
\]
We can prove the equivalence of systems \eqref{90} and \eqref{105v}.

\begin{lemma}
Assume that the functions $\nu_i(t,x)$ satisfy the hyperbolicity condition \eqref{21'} and systems \eqref{90} and \eqref{105v} have common initial data satisfying constraints \eqref{94} for $t=0$. Assume also that the corresponding Cauchy problems for \eqref{90} and \eqref{105v} have a unique classical solution on a time interval $[0,T]$. Then these solutions coincide on  $[0,T]$.
\label{l1}
\end{lemma}

Lemma \ref{l1} can be proved with minor modifications of the proof of corresponding lemma from \cite{ST1}, where a hyperbolic $\varepsilon$-regularization was used for the elliptic system of pre-Maxwell dynamics. We just refer the reader to \cite{ST1} in this connection.

For variable coefficients the counterpart of \eqref{111} reads
\begin{equation}
\nabla_{t,x}\varphi=\hat{a}_1{H}_{N|x_1=0}+\hat{a}_2 {\mathcal{H}}_{N|x_1=0} +\hat{a}_3 {v}_{N|x_1=0}+\hat{a}_0\varphi ,
\label{69v}
\end{equation}
where the vector-functions $\hat{a}_{\alpha}={a}_{\alpha}(\widehat{W}_{|x_1=0})=({a}_{\alpha}^0,{a}_{\alpha}^1,{a}_{\alpha}^2)$ can be easily written down, in particular, $\hat{a}_{3}=(1,0,0)$,
\[
{a}_1^1=\frac{\widehat{\mathcal{H}}_3|_{x_1=0}}{(\widehat{H}_2\widehat{\mathcal{H}}_3-
\widehat{H}_3\widehat{\mathcal{H}}_2)|_{x_1=0}},\quad
{a}_0^1=\frac{(\widehat{\mathcal{H}}_3\partial_1\widehat{H}_N -\widehat{H}_3\partial_1\widehat{\mathcal{H}}_N)|_{x_1=0}}{(\widehat{H}_2\widehat{\mathcal{H}}_3-
\widehat{H}_3\widehat{\mathcal{H}}_2)|_{x_1=0}},\quad \mbox{etc.}
\]
The boundary condition \eqref{91b} together with the result of the substitution of \eqref{69v} into \eqref{91c}--\eqref{91d} is written in the form
\begin{equation}
M
\left(\begin{array}{c}
U\\
V
 \end{array} \right)+b\,\varphi =0\qquad \mbox{on}\ \partial\Omega_T,
\label{rbc}
\end{equation}
where the matrix $M$ and the vector $b$ can be explicitly defined, in particular, the first equation in system \eqref{rbc} is nothing else than \eqref{91b}.

\begin{lemma}
Problem \eqref{89}--\eqref{92} is equivalent to problem \eqref{89}, \eqref{90}, \eqref{rbc}, \eqref{91a}, \eqref{92}.
\label{l2}
\end{lemma}

\begin{proof}
Clearly, smooth enough solutions to problem \eqref{89}--\eqref{92} (if they exist) satisfy problem \eqref{89}, \eqref{90}, \eqref{rbc}, \eqref{91a}, \eqref{92}. We just should prove the opposite. First of all, using \eqref{91a} and the equation for $H$ contained in \eqref{90} (see \eqref{41} with dropped hats), we obtain constraint \eqref{95}. After that we substitute $v_N|_{x_1=0}$ and $H_N|_{x_1=0}$ expressed from \eqref{91a} and \eqref{95} into the second and third boundary  conditions in \eqref{rbc}. Using the result of  this substitution and system \eqref{90} at $x_1=0$ we can derive the second constraint \eqref{96} (we omit calculations). Then \eqref{rbc}, \eqref{95}, \eqref{96} and \eqref{91a} imply the boundary conditions \eqref{91}.
\end{proof}

It is worth noting that, in view of Lemma \ref{l1}, system \eqref{90} in problem \eqref{89}, \eqref{90}, \eqref{rbc}, \eqref{91a}, \eqref{92} can be equivalently replaced by \eqref{105v}. That is, from now on we can concentrate on the proof of the well-posedness of the following problem:
\begin{subequations}
\label{np}
\begin{align}
& \widehat{A}_0\partial_t{U}+\sum_{j=1}^{3}\widehat{A}_j\partial_j{U}+
\widehat{\mathcal{C}}{U}=F \qquad \mbox{in}\ \Omega_T,  \label{89"}
\\
& \varepsilon\mathcal{B}_0\partial_t{V}+\widehat{\mathcal{B}}_1\partial_1{V}+\mathcal{B}_2\partial_2{V}+
\mathcal{B}_3\partial_3{V}=0 \qquad \mbox{in}\ \Omega_T, \label{90"}
\\
& M
\left(\begin{array}{c}
U\\
V
 \end{array} \right)+b\,\varphi =0\qquad  \mbox{on}\ \partial\Omega_T, \label{rbc"}
\\
& ({U},{V})=0\qquad \mbox{for}\ t<0,  \label{92"}
\\
& \partial_t\varphi={v}_{N}-\hat{v}_2\partial_2\varphi-\hat{v}_3\partial_3\varphi +
\varphi\,\partial_1\hat{v}_{N}\qquad  \mbox{on}\ \partial\Omega_T, \label{91a"} \\
&\varphi =0\qquad \mbox{for}\ t<0.\label{92"a}
\end{align}
\end{subequations}

\begin{lemma}
Let the basic state satisfies the assumptions of Theorem \ref{t1}. Then the a priori estimate \eqref{54} holds for problem \eqref{np}.
\label{l3}
\end{lemma}

\begin{proof}
The arguments towards the extension of the a priori estimate \eqref{54} derived in Section \ref{sec:4} for the constant coefficients problem \eqref{99}--\eqref{102} to the case of variable coefficients are similar to those in the relativistic case in \cite{T12}. Then smooth solutions to problem \eqref{89}--\eqref{92} (if they exist) obey estimate \eqref{54}. By virtue of Lemmata \ref{l1} and \ref{l2}, we come to the conclusion of Lemma \ref{l3}.
\end{proof}

To prove the existence of smooth solutions to problem \eqref{np} we can use the idea of \cite{ST1} applied there to  a hyperbolic $\varepsilon$-regularization of the linearized hyperbolic-elliptic plasma-vacuum problem. Namely, assuming that problem \eqref{89"}--\eqref{92"} has a unique smooth solution $(U,V)$ for any given smooth enough function $\varphi$ vanishing in the past, we prove the existence of the solution to \eqref{np} by a fixed point argument. After that we should solve problem \eqref{89"}--\eqref{92"} under the assumption that $\varphi$ is given.

\begin{lemma}
Let for all given ${F} \in H^1_*(\Omega_T)$ and
$\varphi \in H^{3/2}(\partial\Omega_T)$ vanishing in the past  problem \eqref{89"}--\eqref{92"} has a unique solution $(U,V)\in H^1_*(\Omega_T)\times H^1(\Omega_T)$, with $(q,v_N,H_N,V)|_{x_1=0}\in H^{1/2}(\partial\Omega_T)$, such that
\begin{equation}
[{U}]_{1,*,T}+\|{V}\|_{H^{1}(\Omega_T)}+\|(q,v_N,H_N,V)|_{x_1=0}\|_{H^{1/2}(\partial\Omega_T)}\leq C\left\{[F]_{1,*,T} +\|\varphi\|_{H^{3/2}(\partial\Omega_T)}\right\}.
\label{54"}
\end{equation}
Then problem \eqref{np} has a unique solution $(U,V,\varphi)\in H^1_*(\Omega_T)\times H^1(\Omega_T)\times H^{3/2}(\partial\Omega_T)$.
\label{l4}
\end{lemma}

\begin{proof}
Let $\overline\varphi\in H^{3/2}(\partial\Omega_T)$ vanishes in the past. We consider problem \eqref{np} with $\overline\varphi$ instead of $\varphi$ in \eqref{rbc"}. According to our assumption, the exist a unique solution $(U,V)\in H^1_*(\Omega_T)\times H^1(\Omega_T)$, with $(q,v_N,H_N,V)|_{x_1=0}\in H^{1/2}(\partial\Omega_T)$ of \eqref{89"}--\eqref{92"} (with $\overline\varphi$ instead of $\varphi$) enjoying the a priori estimate \eqref{54"} with $\overline\varphi$ instead of $\varphi$. Taking into account the boundary condition \eqref{91a"} and following arguments in \cite{ST1}, we can prove the estimate
\begin{equation}
\|\varphi\|_{H^{3/2}(\partial\Omega_T)}\leq C\left\{[F]_{1,*,T} +\|\overline\varphi\|_{H^{3/2}(\partial\Omega_T)}\right\}.
\label{54"b}
\end{equation}
This defines a map $\overline\varphi\to\varphi$ in $ H^{3/2}(\partial\Omega_T)$. Let $\overline\varphi^1, \overline\varphi^2\in H^{3/2}(\partial\Omega_T)$, and $({U}^1,V^1), ({U}^2,V^2)$, $\varphi^1, \varphi^2 $ be the corresponding solutions of problem \eqref{np} with $\overline\varphi$ instead of $\varphi$ in \eqref{rbc"}, respectively. Thanks to the linearity of the equations we obtain, as for \eqref{54"b},
\[
\|\varphi^1-\varphi^2\|_{H^{3/2}(\partial\Omega_T)}
\leq C \|\overline\varphi^1- \overline\varphi^2\|_{H^{3/2}(\partial\Omega_T)}.
\]
Then the map $\overline\varphi\to\varphi$ has a unique fixed point, by the contraction mapping principle. The fixed point $\overline\varphi=\varphi$  provides a unique solution of problem \eqref{np}.
\end{proof}

Lemma \ref{l4} enables us to consider $-b\varphi$ as a given source term $g$ in \eqref{rbc"}:
\[
M
\left(\begin{array}{c}
U\\
V
 \end{array} \right)=g:=-b\,\varphi \qquad \mbox{on}\ \partial\Omega_T.
\]
Then, following the classical argument, we reduce problem \eqref{89"}--\eqref{92"} to one with homogeneous boundary conditions (with $g=0$) by subtracting from
$(U,V)$ a function $(U',V')\in H^2(\Omega_T)\times H^2(\Omega_T)$ such that
\[
M
\left(\begin{array}{c}
U'\\
V'
 \end{array} \right)=g \qquad \mbox{on}\ \partial\Omega_T.
\]
That is, in view of the above lemmata, assuming that problem \eqref{np} with $\varphi =0$ in \eqref{rbc"} has a unique solution $(U,V)\in H^1_*(\Omega_T)\times H^1(\Omega_T)$, with $(q,v_N,H_N,V)|_{x_1=0}\in H^{1/2}(\partial\Omega_T)$, such that
\begin{equation}
[{U}]_{1,*,T}+\|{V}\|_{H^{1}(\Omega_T)}+\|(q,v_N,H_N,V)|_{x_1=0}\|_{H^{1/2}(\partial\Omega_T)}\leq C[F]_{1,*,T},
\label{54"c}
\end{equation}
we get the solution of problem \eqref{89}--\eqref{92} with the regularity prescribed in Theorem \ref{t1}. In other words, it remains to prove the existence of a unique solution $(U,V)$ to the problem
\begin{subequations}
\label{np1}
\begin{align}
& \widehat{A}_0\partial_t{U}+\sum_{j=1}^{3}\widehat{A}_j\partial_j{U}+
\widehat{\mathcal{C}}{U}=F \qquad \mbox{in}\ \Omega_T,  \label{89""}
\\
& \varepsilon\mathcal{B}_0\partial_t{V}+\widehat{\mathcal{B}}_1\partial_1{V}+\mathcal{B}_2\partial_2{V}+
\mathcal{B}_3\partial_3{V}=0 \qquad \mbox{in}\ \Omega_T, \label{90""}
\\
& M
\left(\begin{array}{c}
U\\
V
 \end{array} \right) =0\qquad  \mbox{on}\ \partial\Omega_T, \label{rbc""}
\\
& ({U},{V})=0\qquad \mbox{for}\ t<0.  \label{92""}
\end{align}
\end{subequations}

\begin{lemma}
Let the basic state satisfies the assumptions of Theorem \ref{t1}. Then problem \eqref{np1} has a unique solution $(U,V)\in H^1_*(\Omega_T)\times H^1(\Omega_T)$, with $(q,v_N,H_N,V)|_{x_1=0}\in H^{1/2}(\partial\Omega_T)$, obeying the a priori estimate \eqref{54"c}.
\label{l5}
\end{lemma}

\begin{proof}
In Section \ref{sec:4} we had constructed the dissipative energy integral which is inequality \eqref{134} and we have the same energy inequality for the case of variable coefficients. However, since integration by parts was used (see Section \ref{sec:4}), we cannot claim that the boundary conditions for the system prolonged up to the first-order tangential derivatives are dissipative in the classical sense. Now we prove that they are indeed dissipative if we drop zero-order terms in $\varphi$.

Setting $\varphi=0$ in \eqref{rbc} means that we drop the lower-order term $\hat{a}_0\varphi$ in \eqref{69v}, i.e., we substitute \eqref{69v} with $\hat{a}_0=0$ into the boundary integrals
\[
2\int_{\mathbb{R}^2}\hat{\mu}\,\partial_{\alpha}\varphi \, \partial_{\alpha}E_1|_{x_1=0}{\rm d}x',
\]
cf. \eqref{110}.\footnote{The same integrals we have for the case of variable coefficients, where $\hat{\mu}=\hat{\mu}(t,x)$ is the function defined through the basic state by formula \eqref{mu}.} For technical simplicity, let us first discuss the case of constant coefficients when, in particular,  $\hat{\mu}={\rm const}$. The substitution of the second line in \eqref{111} into the above boundary integral with $\alpha =2$ leads to the appearance of, for example, the following integral
\begin{equation}
2\int_{\mathbb{R}^2}\hat{\mu}\,a^1_2\,\mathcal{H}_1  \partial_2E_1|_{x_1=0}{\rm d}x'.
\label{i1}
\end{equation}

On the one hand, we can rewrite the boundary integral in \eqref{i1} by passing to the volume integral and integrating by parts:
\[
\int_{\mathbb{R}^2}\mathcal{H}_1  \partial_2E_1|_{x_1=0}{\rm d}x'=\int_{\mathbb{R}^3_+}\left( \partial_2\mathcal{H}_1  \partial_1E_1-\partial_1\mathcal{H}_1  \partial_2E_1\right){\rm d}x,
\]
and, on the other hand, we have
\[
\int_{\mathbb{R}^2}\mathcal{H}_1  \partial_2E_1|_{x_1=0}{\rm d}x'=\int_0^t\int_{\mathbb{R}^2}\left( \partial_{\tau}\mathcal{H}_1  \partial_2E_1-\partial_2\mathcal{H}_1  \partial_{\tau}E_1\right)|_{x_1=0}{\rm d}\tau{\rm d}x'.
\]
Hence, integral \eqref{i1} disappears in \eqref{110} with $\alpha=2$ after the addition to it the identity
\begin{equation}
\int_{\mathbb{R}^3_+}\left( \partial_2\mathcal{H}_1  \partial_1E_1-\partial_1\mathcal{H}_1  \partial_2E_1\right){\rm d}x -\int_0^t\int_{\mathbb{R}^2}\left( \partial_{\tau}\mathcal{H}_1  \partial_2E_1-\partial_2\mathcal{H}_1  \partial_{\tau}E_1\right)|_{x_1=0}{\rm d}\tau{\rm d}x'=0
\label{i2}
\end{equation}
multiplied by $2\hat{\mu} a^1_2$, and the volume integral
\[
2\int_{\mathbb{R}^3_+}\hat{\mu} a^1_2\left( \partial_2\mathcal{H}_1  \partial_1E_1-\partial_1\mathcal{H}_1  \partial_2E_1\right){\rm d}x
\]
makes a corresponding contribution to the integral of the quadratic form with the matrix $\hat{\mu}Q$ in \eqref{121}.

We can write down a symmetric (but not hyperbolic) system for which we have the energy identity \eqref{i2}. Indeed, the system of evident equations
\begin{align*}
 -\partial_t(\partial_2E_1)+\partial_2(\partial_tE_1)=0,\quad
\partial_t(\partial_1E_1)-\partial_1(\partial_tE_1)=0,\quad
\partial_1(\partial_2E_1)-\partial_2(\partial_1E_1)=0,\\
 \partial_t(\partial_2\mathcal{H}_1)-\partial_2(\partial_t\mathcal{H}_1)=0,\quad
-\partial_t(\partial_1\mathcal{H}_1)+\partial_1(\partial_t\mathcal{H}_1)=0,\quad
-\partial_1(\partial_2\mathcal{H}_1)+\partial_2(\partial_1\mathcal{H}_1)=0
\end{align*}
is the symmetric system
\begin{equation}
B_3\partial_tX- B_1\partial_1X +B_2\partial_2X=0 \label{ts}
\end{equation}
for the vector $X=(\partial_1\mathcal{H}_1,\partial_2\mathcal{H}_1,\partial_t\mathcal{H}_1,\partial_1E_1,\partial_2E_1,\partial_tE_1)$ and obeys the energy identity \eqref{i2}, where the symmetric matrices $B_j$ are the same as in the Maxwell equations \eqref{25}. Moreover, using the divergence constraints \eqref{103},
\[
\partial_1\mathcal{H}_1 =\partial_2\mathcal{H}_2+\partial_3\mathcal{H}_3,\quad \partial_1E_1=\partial_2E_2+\partial_3E_3,
\]
 we can pass in \eqref{ts} from the vector $X$ to the vector of only tangential derivatives $Y=(\partial_2\mathcal{H}_2,\partial_3\mathcal{H}_3,\partial_2\mathcal{H}_1,\partial_t\mathcal{H}_1,\partial_2E_2,\partial_3E_3,\partial_2E_1,\partial_tE_1)$ keeping the symmetry property:
\begin{equation}
(\mathcal{T}^TB_3\mathcal{T})\partial_tY- (\mathcal{T}^TB_1\mathcal{T})\partial_1Y +(\mathcal{T}^TB_2\mathcal{T})\partial_2Y=0 ,
\end{equation}
where $X=\mathcal{T}Y$, the rectangular matrix $\mathcal{T}$ can be easily written down, and the matrices $\mathcal{T}^TB_j\mathcal{T}$ are again symmetric.

Using analogous simple arguments, we can understand that any integration by parts in Section \ref{sec:4} can be associated with the addition of the energy identity for some symmetric system for a vector which components are components of the vector $Z$ of tangential derivative of $U$ and $V$. It means that we add to the symmetric hyperbolic system
\begin{equation}
\mathfrak{A}_0\partial_tZ+\sum_{j=1}^3\mathfrak{A}_j\partial_jZ=\mathfrak{F}
\label{bsys}
\end{equation}
constructed from system \eqref{89""}, \eqref{90""} (here in the case of constant coefficients) a symmetric system
\[
\hat{\mu}{\rm Q}_0\partial_tZ+\sum_{j=1}^3{\rm Q}_j\partial_jZ={\rm Q}_4\mathfrak{F}
\]
and then consider the energy identity for the resulting system
\begin{equation}
(\mathfrak{A}_0+\hat{\mu}{\rm Q}_0)\partial_tZ+\sum_{j=1}^3(\mathfrak{A}_j+{\rm Q}_j)\partial_jZ=(I+{\rm Q}_4)\mathfrak{F}
\label{i3}
\end{equation}
which is hyperbolic under condition \eqref{122}, where the matrices $\mathfrak{A}_0$ and ${\rm Q}_0$ were defined in Section \ref{sec:4},
\[
\mathfrak{A}_1={\rm diag}(\widehat{A}_1,\varepsilon^{-1}\widehat{\mathcal{B}}_1,\ldots ,\widehat{A}_1,\varepsilon^{-1}\widehat{\mathcal{B}}_1),\quad
\mathfrak{A}_k={\rm diag}(\widehat{A}_k,\varepsilon^{-1}\mathcal{B}_k,\ldots ,\widehat{A}_k,\varepsilon^{-1}\mathcal{B}_k),\quad k=2,3,
\]
\[
\mathfrak{F}=(\partial_tF,\overline{0},\partial_2F,\overline{0},\partial_3F,\overline{0}),\quad \overline{0}=(0,0,0,0,0,0),
\]
and the matrices ${\rm Q}_i$ ($i=\overline{1,4}$) can be explicitly written down if necessary.

If we do not apply the Young inequality towards the derivation of inequality \eqref{121}, we obtain the corresponding identity
\begin{equation}
\int_{\mathbb{R}^3_+}\left( (\mathfrak{A}_0+ \hat{\mu}{\rm Q}_0)Z,Z\right){\rm d}x
-2\int_{\mathbb{R}^3_+}\hat{\mu}F_1\partial_tE_1{\rm d}x
=2\int_0^t\int_{\mathbb{R}^3_+}(\mathfrak{F},Z){\rm d}x{\rm d}\tau .
\label{121i}
\end{equation}
Instead of the usage of inequality \eqref{124} we can first express $\partial_tE_1$ from the fourth equation in \eqref{100} by taking into account the third constraint in \eqref{103},
\[
\partial_tE_1 = -\varkappa (\partial_2E_2+\partial_3E_3) +\varepsilon^{-1}(\partial_3\mathcal{H}_2-\partial_2\mathcal{H}_3),
\]
and then we have
\begin{multline*}
\int_{\mathbb{R}^3_+}\hat{\mu}F_1\partial_tE_1{\rm d}x=\int_0^t\int_{\mathbb{R}^3_+}\hat{\mu}\bigl\{\varkappa (\partial_2F_1\partial_{\tau}E_2 +\partial_3F_1\partial_{\tau}E_3)\\
-\varepsilon^{-1}( \partial_3F_1\partial_{\tau}\mathcal{H}_2- \partial_2F_1\partial_{\tau}\mathcal{H}_3) \bigr\} {\rm d}x{\rm d}\tau
:=\int_0^t\int_{\mathbb{R}^3_+}({\rm Q}_4\mathfrak{F},Z){\rm d}x{\rm d}{\tau},
\end{multline*}
and \eqref{121i} implies the energy identity
\begin{equation}
\int_{\mathbb{R}^3_+}\left( (\mathfrak{A}_0+ \hat{\mu}{\rm Q}_0)Z,Z\right){\rm d}x
=2\int_0^t\int_{\mathbb{R}^3_+}((I+{\rm Q}_4)\mathfrak{F},Z){\rm d}x{\rm d}\tau .
\label{121ii}
\end{equation}
Clearly, here the matrix ${\rm Q}_4$ is the same as in system \eqref{i3} for which the energy identity reads
\begin{multline}
\int_{\mathbb{R}^3_+}\left( (\mathfrak{A}_0+ \hat{\mu}{\rm Q}_0)Z,Z\right){\rm d}x-
\int_0^t\int_{\mathbb{R}^3_+}((\mathfrak{A}_1+{\rm Q}_1)Z,Z)|_{x_1=0}{\rm d}x'{\rm d}\tau \\
=2\int_0^t\int_{\mathbb{R}^3_+}((I+{\rm Q}_4)\mathfrak{F},Z){\rm d}x{\rm d}\tau .\qquad
\label{121iii}
\end{multline}
Comparing \eqref{121ii} and \eqref{121iii} we conclude that the boundary conditions for system \eqref{i3} are {\it dissipative}, to be exact,
\[
((\mathfrak{A}_1+{\rm Q}_1)Z,Z)|_{x_1=0} =0.
\]

For the case of variable coefficients some ``lower-order'' terms appear in energy integrals, but again any integration by parts towards the proof of the energy inequality \eqref{134} can be associated with the addition of some symmetric system for the vector $(W,Z)$, with $W=(U,V)$ (unlike the case of constant coefficients it will contain also equations for the original unknown $W$). The boundary integral will again disappear in the final energy identity, i.e., the boundary conditions for the prolonged system are dissipative. Moreover, for guaranteeing the hyperbolicity of the prolonged system for $(W,Z)$, together with the variable coefficients counterpart of \eqref{bsys} we consider the ``trivial'' system
\begin{equation}
\beta\partial_tW -\beta W_0=0
\label{i4}
\end{equation}
with the constant $\beta$ large enough and the vector $W_0$ containing in $Z=(W_0,W_2,W_3)$.

The vector  $Z$ satisfies the variable coefficients counterpart of \eqref{bsys} obtained by the formal differentiation of system \eqref{89""}, \eqref{90""} with respect to $t$, $x_2$ and $x_3$, i.e., $W_2$ and $W_3$ are associated with $\partial_2W$ and $\partial_3W$ respectively. Since zero-order terms in $\varphi$ play no role in the derivation of the energy estimate for the variable coefficients problem (see \cite{T12}), problem \eqref{np1} obeys the a priori estimate \eqref{54"c}. Moreover, the same estimate takes place for $\partial_2W$ and $\partial_3W$. But, since $W_2$ and $W_3$ satisfies the same problem as $\partial_2W$ and $\partial_3W$, estimate \eqref{54"c} implies uniqueness, i.e., $W_k=\partial_kW$, $k=2,3$.

Thus, we obtain a symmetric (prolonged) system by adding to \eqref{i4} and the variable coefficients counterpart of \eqref{bsys} the symmetric system for $(W,Z)$ associated with integrations by parts. This system is hyperbolic under condition \eqref{122} and for the constant $\beta$ large enough. Moreover, the boundary conditions for this system are dissipative. Finally, as the boundary is characteristic of constant multiplicity \cite{Rauch}, we may apply the result of \cite{secchi95,secchi96} and we get the solution from  $H^m_*(\Omega_T)$, with $m\geq 1$ and the noncharacteristic unknowns having a greater degree of regularity in the normal direction. In view of \eqref{i4}, integrating the interior equations and the boundary conditions for $W_0=\partial_tW$ (containing lower-order terms with $W$) over the time interval $[0,t]$ and taking into account the zero initial data, we get the solution of \eqref{np1} with the prescribed regularity.
\end{proof}

Taking into account the arguments before Lemma \ref{l5}, the proof of this lemma completes the proof of the existence of the solution of problem \eqref{89}--\eqref{92}.

\section*{Achnowledgments}

This work was supported by RFBR (Russian Foundation for Basic Research) grant No. 10-01-00320-a.

\end{document}